\documentclass[journal]{IEEEtran}
\usepackage{subfigure}
\usepackage{comment}
\usepackage{multicol}
\usepackage{url}
\usepackage{pst-pdf}
\usepackage{epsfig}
\usepackage{pst-grad} 
\usepackage{pst-plot} 
\usepackage{amsmath,epsfig}
\usepackage{amssymb}
\usepackage{psfrag}
\usepackage{color}
\usepackage{pstricks}
\usepackage{cite}
\usepackage{algorithm}
\usepackage{epsfig}
\usepackage{amsthm}
\usepackage{psfrag}
\usepackage{color}
\usepackage{float}
\usepackage{stfloats}
\usepackage{wrapfig}
\usepackage{graphicx}
\usepackage{blindtext}
\usepackage{pstricks,pst-plot}
\usepackage{pst-bezier}
\usepackage{pst-math}
\usepackage{setspace}
\usepackage{mathtools, cuted}

\usepackage{changepage}   

\usepackage{bm}
\usepackage{hyperref}
\usepackage{datetime}

\def\bZ{{\bf Z}}

\def\by{{\bf y}}
\def\bY{{\bf Y}}
\def\bV{{\bf V}}
\def\bW{{\bf W}}

\def\bh{{\bf h}}

\def\bH{{\bf H}}

\def\bC{{\bf C}}

\def\bR{{\bf R}}

\def\bI{{\bf I}}

\def\bA{{\bf A}}

\def\bx{{\bf x}}

\def\bLambda{{\bf \Lambda}}

\theoremstyle{plain}\newtheorem{athm}{Theorem}
\theoremstyle{plain}
\AtBeginDocument{}
\begin{document}

\title{On the Shift Operator, Graph Frequency and Optimal Filtering in Graph Signal Processing}
\author{Adnan Gavili and Xiao-Ping Zhang, ~\IEEEmembership{Senior~member,~IEEE}
\thanks{This work was supported in part by the Natural Sciences and Engineering Research Council of Canada (NSERC), Grant No. RGPIN239031.}
\thanks{%
The authors are with the Department of Electrical and Computer Engineering,
Ryerson University, 350 Victoria Street, Toronto, Ontario, Canada
M5B 2K3. E-mail: adnan.gavili@ryerson.ca and xzhang@ee.ryerson.ca.
Xiao-Ping Zhang is the corresponding author.}
\thanks{Version: \today} 
}

\maketitle

\begin{abstract}
Defining a sound shift operator for signals existing on a certain graph structure, similar to the well-defined shift operator in classical signal processing, is a crucial problem in graph signal processing, since almost all operations, such as filtering, transformation, prediction, are directly related to the graph shift operator. We define a set of energy-preserving shift operators that satisfy many properties similar to their counterparts in classical signal processing. Our definition of the graph shift operator negates the shift operators defined in the literature, such as the graph adjacency matrix and Laplacian matrix based shift operators, which modify the energy of a graph signal. We decouple the graph structure represented by eigengraphs and the eigenvalues of the adjacency matrix or the Laplacian matrix. We show that the adjacency matrix of a graph is indeed a linear shift invariant (LSI) graph filter with respect to the defined shift operator. We introduce graph finite impulse response (GFIR) and graph infinite impulse response (GIIR) filters and obtain explicit forms for such filters. We further define autocorrelation and cross-correlation functions of signals on the graph, enabling us to obtain the solution to the optimal filtering on graphs, i.e., the corresponding Wiener filtering on graphs and the efficient spectra analysis and frequency domain filtering in parallel with those in classical signal processing. This new shift operator based GSP framework enables the signal analysis along a correlation structure defined by a graph shift manifold as opposed to classical signal processing operating on the assumption of the correlation structure with a linear time shift manifold.  We further provide the solution to the optimal linear predictor problem over general graphs. Several illustrative simulations are presented to validate the performance of the designed optimal LSI filters.
\end{abstract}


\section{Introduction} \label{IntroSection}
Graph signal processing (GSP) is an emerging field, focusing on representing signals as evolving entities on graphs and analyzing the signals based on the structure of the graph \cite{Ortega,Moura1,Moura2,Moura3}. The temporally evolving measured data from variety of sources in a network, such as the measured data from sensors in wireless sensor networks, body area sensor networks, transportation networks and weather networks, are compatible with signal representation on certain graphs. For instance, a network of sensors implanted in a human body to measure the temperatures of different tissues can be viewed as a graph in which the sensor nodes are the graph nodes and the graph structure shows the connection between the sensor nodes. Moreover, the measured temperatures by the nodes are the signals existing on the corresponding graph. Hence, GSP can be a powerful tool for analyzing and interpreting such signals existing on graphs. 

Classical signal processing has provided a wide range of tools to analyze, transform and reconstruct signals regardless of the true nature of the signals evolution. Indeed, classical signal processing may not provide an effective way to represent and analyze the signals that exist on a graph structure. GSP is an attempt to develop a universal tool to process signals on graphs. More specifically, GSP benefits from algebraic and graph theoretic concepts, such as graph spectrum and graph connectivity, to analyze structured data \cite{Ortega,Puschel1,Puschel2}.

Two major approaches have been developed for signal processing on graphs. The first approach is to use the graph Laplacian matrix as the underlying building block for the definitions and tools in GSP \cite{Ortega}. The second approach is to use the adjacency matrix of the underlying graph as the shift operator on graph \cite{Moura1,Moura2,Moura3}. Both approaches define fundamental signal processing concepts on graphs, such as filtering, transformation, downsampling.

Graph wavelet transforms are discussed in \cite{Leonardi,COIFMAN200653,HAMMOND2011129}. 
The idea of graph filter banks is developed in \cite{Ortega2} with the design of critically-sampled wavelet-filter banks on graphs. 
Authors in \cite{Ekambaram2} introduce two-channel (low-pass and high-pass), critically-sampled, perfect-reconstruction filterbanks for signals defined on circulant graphs. 
The authors in \cite{Kotzagiannidis} extend the framework of sampling and reconstructing signals with a finite rate of innovation (FRI) to the graph domain.
Authors in \cite{Kotzagiannidis2} present novel families of wavelets and associated filterbanks for the analysis and representation of functions defined on circulant graphs and generalize to arbitrary graphs in the form of graph approximations. In \cite{Ekambaram}, the authors present a method to decompose an arbitrary graph or filter into a combination of circulant structures. 
In \cite{XDong,XDong2,loh2013}, the authors focus on recovering the graph structure, i.e., the graph adjacency matrix, via formulating a design problem. The obtained graph structure can then be used to obtain the graph Fourier basis, the new graph shift operator and graph filters. In this paper, we assume that the graph structure is already obtained, e.g., using any of these approaches, and we aim to define a graph shift operator that satisfies certain properties.

When the structure of a graph is known, the common effort in GSP is to define a \textit{shift operator} on the graph and then introduce the concepts of filtering, transformation, denoising, prediction, compression and other operations similar to the conventional counterparts in classical signal processing, based on the shift operator. 
It is defined in \cite{Ortega} as the translation on graph via generalized convolution with a delta centered at vertex $n$. In \cite{Moura1}, the graph shift operator is the adjacency matrix of the graph and simple justification of such a choice is presented. However, none of these operators satisfy the energy-preserving property similar to their counterpart in classical signal processing. More specifically, applying the shift operator in \cite{Ortega,Moura1} to a graph signal several times will change the energy content of the graph signal and its frequency components, making it difficult to justify and design the filter frequency response as in classical signal processing. An isometric shift operator has recently been introduced in \cite{GraphTranslation,StationaryGraphSignal,GiraultThesis}, which satisfies the energy-preserving property. This shift operator is a matrix whose eigenvalues are derived from the graph Laplacian matrix. The limitation of this approach is that its phase shifts are structure-dependent and do not satisfy some other desired properties leading to computationally efficient spectral analysis.

Motivated by the graph shift matrix defined in \cite{Moura1}, but fundamentally different, we introduce a unique set of graph shift operators that satisfy the properties of the shift operator in classical signal processing. The new shift operator preserves the energy content of the graph signal in the frequency domain. We essentially decompose the graph adjacency matrix (the Laplacian matrix can be handled the same way) into two parts. The first part is the graph structure part represented by eigengraphs associated with frequency components of a graph. The second part is the filtering part represented by the eigenvalues of the adjacency matrix, which changes the amplitude of the frequency components. The eigenvalues of the new shift operators therefore only represent phase shift of frequency components that can be flexibly constructed. A special construction of these phase shift eigenvalues with nice properties is also given. 

We then elaborate on the structure of \textit{linear shift invariant} (LSI) graph filters and show that any adjacency matrix can indeed be written as an LSI graph filter using the presented new shift operator. Furthermore, we define the \emph{graph finite impulse response} (GFIR) and \emph{graph infinite impulse filters} (GIIR), similar to the classical signal processing counterparts, and obtain an explicit form for such filters. Based on the defined shift operator, we introduce autocorrelation and cross-correlation functions on graph. We then formulate the optimal filtering and spectrum analysis on graph, i.e., the corresponding Wiener-Hopf equation and \emph{Wiener} filtering on graphs, and obtain the structure of such filters for any arbitrary graph structure. We finally elaborate on the best linear predictor graph filters and provide several illustrative simulation setups to verify the performance improvements of optimal filtering using our new graph shift operator.

The contribution of this paper can be summarized as follows:

	\begin{itemize}
		\item We define a general set of graph shift operators that satisfy the energy-preserving property in the frequency domain and other properties in classical signal processing. These shift operators only change the phase of frequency components. Especially, we design a specific shift operator with the desired periodicity property as in classical signal processing. The shift operation can then be considered as discrete-time lossless information flowing structure on a graph.
		
		\item For a given graph, we construct a set of \emph{eigengraphs} that represent basic correlation structures of a graph frequency component. When applied on any graph signal, each eigengraph is a projection operator that projects the signal to a single graph frequency component subject to only a phase shift. The new shift operator is a linear combination of eigengraphs. \label{eigengraph1}

		\item We investigate the properties of the presented shift operator for \emph{linear shift invariant filtering} and show that the adjacency matrix is indeed a LSI filter based on our new graph shift operator. 
		
		\item We define autocorrelation and cross-correlation functions of a signal on graph. We then obtain a closed-form solution to the \emph{Wiener} filtering problem and show that it has efficient power-spectrum representation in certain graphs similar to classical signal processing. Such a power spectral analysis can only be obtained using our new shift operator. This new shift operator based GSP framework enables the signal analysis along a correlation structure defined by a graph shift manifold as opposed to classical signal processing operating on the assumption of the correlation structure with a linear time shift manifold.  
	\end{itemize}

The paper is organized as follows. In section II, we discuss the basics of GSP and present a new set of shift operators. Section III introduces graph filters and Fourier transforms based on the new shift operator.  We derive the optimal LSI graph filters in section IV. Section V presents the simulations and section VI concludes of the paper.

\emph{Notations:}  Matrices and vectors are represented by uppercase and lowercase boldface letters, respectively. Transpose and Hermitian (conjugate transpose) operations are represented by $(\cdot)^T$ and $(\cdot)^H$, respectively. The notation $\bI$ stands for the identity matrix, and $\circledast$ and $\ast$ are the circular and aperiodic convolution operators, respectively.

 \section{A New Set of Shift Operators and Graph Frequency Components}

\subsection{Signals on Graph}\label{sec-siggraph}

Consider a dataset with $N$ distinct elements, where some information regarding the relations between data elements is available. One can represent such a dataset and the corresponding relational information as a graph. A graph can be denoted by a $G=\{ \mathcal{V},\bf{A} \}$, where $\mathcal{V}=\{ \nu_0, \cdots \nu_{N-1} \}$ is the set of all vertices of the graph, representing the elements in the dataset, and $\bf{A}$ is the weighted adjacency matrix that represents the relation between nodes. More specifically, if there is a relation between nodes $\nu_n$ and $\nu_m$, then $a_{n,m}={\bf{A}}(n,m)\neq0$, otherwise $a_{n,m}=0$. We note that the elements of the adjacency matrix $\bf{A}$ is not restricted to a specific set of values. In this paper, we consider a general graph with real-valued adjacency matrix $\bA$, either directed or undirected, and assume that data elements take complex scalar values. We define a \emph{graph signal} as a one-to-one mapping from the set of all vertices to the set of complex numbers:
\begin{align}
{\bf x}\; : \; &\mathcal{V}\;\rightarrow \;\mathbb{C} , \nu_n \rightarrow x_n.
\end{align} 
Without loss of generality, we represent a graph signal as a vector whose elements are complex numbers assigned to the nodes, ${\bf x}=[x_1, \cdots , x_N]^T$, where $T$ stands for the transpose operator.

As a special case, a directed cyclic graph is shown in Fig.~\ref{fig:timeseries}. 
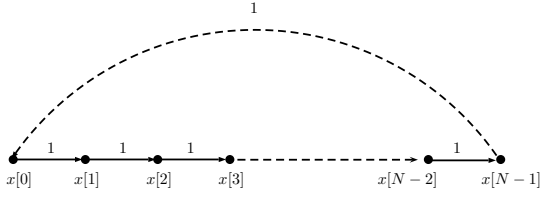
\begin{figure} 
\begin{center}
\scalebox{0.6} 
{
\begin{pspicture}(1,2.976533)(13.132351,6.5438094)
\psdots[dotsize=0.2](2.7619047,3.647619)
\psdots[dotsize=0.2](4.3619046,3.647619)
\psdots[dotsize=0.2](5.9619045,3.647619)
\psdots[dotsize=0.2](11.9619055,3.647619)
\psarc[linewidth=0.04,linestyle=dashed,dash=0.17638889cm 0.10583334cm,arrowsize=0.05291667cm 1.6,arrowlength=1.4,arrowinset=0.0]{->}(6.5238094,0.0){6.5238094}{34.938507}{145.84827}
\psdots[dotsize=0.2](1.1619047,3.647619)
\psdots[dotsize=0.2](10.361904,3.647619)
\psline[linewidth=0.04,linestyle=dashed,dash=0.17638889cm 0.10583334cm,arrowsize=0.05291667cm 2.0,arrowlength=1.4,arrowinset=0.4]{->}(6.1295385,3.6438093)(10.129539,3.6438093)
\usefont{T1}{ptm}{m}{n}
\rput(1.3,3.1980953){$x[0]$}
\usefont{T1}{ptm}{m}{n}
\rput(2,3.9){$1$}
\usefont{T1}{ptm}{m}{n}
\rput(2.8,3.1980953){$x[1]$}
\usefont{T1}{ptm}{m}{n}
\rput(3.6,3.9){$1$}
\usefont{T1}{ptm}{m}{n}
\rput(4.4,3.1980953){$x[2]$}
\usefont{T1}{ptm}{m}{n}
\rput(5.1,3.9){$1$}
\usefont{T1}{ptm}{m}{n}
\rput(6,3.1980953){$x[3]$}
\usefont{T1}{ptm}{m}{n}
\rput(11,3.9){$1$}
\usefont{T1}{ptm}{m}{n}
\rput(9.9,3.1980953){$x[N-2]$}
\usefont{T1}{ptm}{m}{n}
\rput(6.5,7){$1$}
\usefont{T1}{ptm}{m}{n}
\rput(12.2,3.1980953){$x[N-1]$}
\psline[linewidth=0.04,arrowsize=0.05291667cm 1.6,arrowlength=1.4,arrowinset=0.0]{->}(1.220852,3.6497245)(2.6945362,3.6497245)
\psline[linewidth=0.04,arrowsize=0.05291667cm 1.6,arrowlength=1.4,arrowinset=0.0]{->}(2.8450625,3.6497245)(4.3187466,3.6497245)
\psline[linewidth=0.04,arrowsize=0.05291667cm 1.6,arrowlength=1.4,arrowinset=0.0]{->}(10.38401,3.6297245)(11.857694,3.6297245)
\psline[linewidth=0.04,arrowsize=0.05291667cm 1.6,arrowlength=1.4,arrowinset=0.0]{->}(4.4050627,3.6497245)(5.878747,3.6497245)
\end{pspicture} 
}
\caption{Graph representation of time series periodic data.}
    \label{fig:timeseries}
\end{center}
\end{figure}
Such a graph is compatible with the graph representation of a periodic time series signal, $x[n]=x[n+N]$, with $N$ signal points, i.e., one can assign a node to each signal point and the relation between the signal points is the time shift. One can easily show that the graph adjacency matrix for the directed cyclic graph  is given by

\begin{equation}\label{Circulant}
{{\bf A}}={\bf C} = 
 \begin{pmatrix}
  0 & 0 & \cdots & 1 \\
  1 & 0 & \cdots & 0 \\
  \vdots  & \vdots  & \ddots & \vdots  \\
  0 &  \cdots & 1 & 0 
 \end{pmatrix}.
\end{equation}
There are two fundamental components in GSP: the signals represented by the values on vertices, and the signal correlation structure represented by the connections between vertices.

\subsection{Graph Shift Operator, Information Flow and Filtering} \label{sec-gso}

A graph shift operator allows us to define the notion of information flow over a graph. Indeed, it represents one elementary discrete step on how the information propagates (shifts) from one node to its neighbors. In classical signal processing, i.e., the case where the graph structure is a cyclic graph, the information flow is unidirectional, i.e., from each node to only its next neighbor. In a more complicated graph structure, the information flow will neither be restricted to unidirectional structure nor to a limited number of physical neighbors but depend on the graph adjacency matrix. Therefore, the notion of shift operator on graph must be clarified.

\subsubsection{Graph shift}

In \cite{Moura2}, the notion of shift operator is defined as \emph{a local operation that replaces a signal value at each node of a graph with the linear combination of the signal values at the neighbors of that node}. Shift operator is a fundamental element in digital signal processing. Specifically, for a shift operator $\bf \Phi$ and a graph signal $\bf x$, the one-step shifted version of the graph signal, which is a new graph signal, is $\bf \Phi x$. And the $n$-step shifted version of the signal is ${\bf \Phi}^n x$.

\subsubsection{Graph signal state change}
	
A graph shift operator is a linear operator such that when it applies to a graph signal ${\bf x}_n$ at state (or time) $n$, it changes the graph signal into a new graph signal ${\bf x}_{n+1}$ at step $n+1$. We use the index $n$ to show the state of the graph signal. Equivalently, when a graph shift operator is applied to a graph signal, the state of the original graph signal is shifted to a new state by one unit of shift. 

Note that this is similar to time series analysis, when a signal $x(t)$ is shifted in time by a certain amount $T$, $y(t)=x(t-T)$, the signal $x(\cdot)$ at time-state $t-T$ will be mapped to $y(\cdot)$ at time-state $t$. The state of the signal at a certain time stamp is updated by the shift operator. Similarly, we defined the $(n+1)$-th state of the vector of the graph signal, i.e., ${\bf x}_n$, as the $n$-th shifted version of ${\bf x}_0$. For instance, when the graph shift operator applies to a graph signal at state $n$, i.e., ${\bf x}_n$, and it changes the graph signal to ${\bf x}_{n+1}$. More specifically, ${\bf x}_{n+1}={\bf \Phi}{\bf x}_n$.

Also note that in time series analysis, a time shift corresponds to a local shift in the cyclic graph. For a general graph shift, there is no such straightforward relationship.

\subsubsection{Graph signal filtering}

A linear filtering is defined by a matrix operation on the graph signal such that the result is also a graph signal. If we define the filter matrix as $\bf H$, the filtered graph signal can be written as $\bf H x$. We will show later that if the filtering operation also satisfies the shift invariance property, the filter can be written as a polynomial of the new graph shift operator and the filter operation is indeed a modification of the amplitudes of existing signal frequency components, as in classical signal processing.

\subsection{The New Graph Shift Operator}\label{sec-newgso}

We now define a set of \emph{energy-preserving shift operators} for an arbitrary graph structure. 

${\bf Definition}$ : Given the adjacency matrix ${\bf A}$ for an arbitrary graph, assume that it is diagonalizable and its eigen decomposition is ${\bf A}={\bf V}{\boldsymbol \Lambda}{\bf V}^{-1} = \sum_{i=1}^N\lambda_i{\bf v}_i \tilde{\bf v}_i^T=\sum_{i=1}^N\lambda_i \hat{\bf V}_i$, where ${\bf V}=[{\bf v}_1 \; {\bf v}_2 \; \cdots \; {\bf v}_N]$ and $({\bf V}^{-1})^T=[\tilde{\bf v}_1 \; \tilde{\bf v}_2 \; \cdots \; \tilde{\bf v}_N]$, and ${\bf v}_i$ and $\tilde{\bf v}_i$ are $N \times 1$ column vectors of ${\bf V}$ and $({\bf V}^{-1})^T$, respectively. We define the matrix ${\bf A}_{\phi}={\bf V}{\boldsymbol \Lambda}_\phi{\bf V}^{-1}=\sum_{i=1}^N\lambda_{\phi_i} \hat{\bf V}_i$ to be the shift operator with
\begin{equation}\label{eq-lambda_phi}
{\boldsymbol \Lambda}_{\phi}=\text{diag}(\lambda_{\phi_1},\lambda_{\phi_2},\cdots,\lambda_{\phi_N}),
\end{equation}
where $\lambda_{\phi_k}=e^{j\phi_k}$, $\phi_k$ is an arbitrary phase in $[0,2 \pi]$ where $\phi_k\neq \phi_l$ for $k\neq l$, $|{\boldsymbol \Lambda}_\phi|={\bf I}$, $|\cdot|$ is defined as the point-wise absolute value operator. Thus,
\begin{align}
{\bf A}={\bf V}{\boldsymbol \Lambda}{\bf V}^{-1}&={\bf V}{\boldsymbol \Lambda}_{h}{\boldsymbol \Lambda}_{\phi}{\bf V}^{-1} \nonumber \\  
&={\bf V}{\boldsymbol \Lambda}_h{\bf V}^{-1}{\bf V}{\boldsymbol \Lambda}_{\phi}{\bf V}^{-1} \nonumber \\
&={\bf A}_h{\bf A}_{\phi}={\bf A}_{\phi}{\bf A}_h,
\end{align}
 where ${\boldsymbol \Lambda}_h={\boldsymbol \Lambda}{\boldsymbol \Lambda}_{\phi}^{-1}$ and ${\bf A}_{h}={\bf V}{\boldsymbol \Lambda}_{h}{\bf V}^{-1}$.
 In essence, the shift operator ${\bf A}_{\phi}$ preserves all the eigenvectors of the adjacency matrix ${\bf A}$, but replaces all the eigenvalues of  ${\bf A}$ with pure phase shifts.
 
$\bf Definition$: We further define a special new shift operator as 
\begin{equation}\label{eq-Ae}
{\bf A}_e={\bf V}{\boldsymbol \Lambda}_e{\bf V}^{-1}, \lambda_{e_k}\lambda_{e_l}^*=e^{-j\frac{2\pi (k-l)}{N}}, \forall k,l=1,\cdots,N,
\end{equation}
where ${\boldsymbol \Lambda}_{e}=\text{diag}(\lambda_{e_1},\lambda_{e_2},\cdots,\lambda_{e_N})$.
One can write
\begin{align}\label{LambdaEquation} 
\lambda_{e_k}=e^{j(\phi_{\text{const}}+\frac{-2 \pi (k-1)}{N})},
\end{align}
where $\phi_{\text{const}}$ can be any arbitrary constant phase shift. Without loss of generality, we will assume $\phi_{\text{const}}=0$ in the rest of this paper. The shift operator ${\bf A}_e$ and ${\bf A}_\phi$ satisfies the following properties:

\emph{Property 1}: $\| ({\bf A}_\phi^k {\bf x})_{\cal F}\|^2  = \left\|  {\bf x}_{\cal F}\right\|^2 $, where ${\bf x}_{\cal F}$ is frequency representation of graph signal defined by ${\bf x}_{\cal F}={\bf V}^{-1}{\bf x}$.

\begin{proof}
 Note that $k$-th shifted version of the graph signal in the Fourier domain is 
\begin{align}\label{eq-p1a}
({\bf A}_{\phi}^k{\bf x})_{\cal F}={\bf V}^{-1}{\bf A}_{\phi}^k{\bf x}={\bf V}^{-1}({\bf V}{\boldsymbol \Lambda}_{\phi}^k{\bf V}^{-1}){\bf x}={\boldsymbol \Lambda}_{\phi}^k{\bf x}_{\cal F}.
\end{align}
Thus its energy
\begin{align}\label{eq-p1b}
\|({\bf A}_\phi^k {\bf x})_{\cal F}\|_2^2=||{\boldsymbol \Lambda}_{\phi}^k{\bf x}_{\cal F}||_2^2=({\boldsymbol \Lambda}_{\phi}^k{\bf x}_{\cal F})^H{\boldsymbol \Lambda}_{\phi}^k{\bf x}_{\cal F}=\left\|  {\bf x}_{\cal F}\right\|^2,
\end{align}
i.e., the energy of the graph signal in the frequency domain for any amount of shift is constant. 
\end{proof}

Note that for a unitary graph Fourier operator, ${\bf V}^{-1} = \bV^H$, then $\left\|  {\bf x}_{\cal F}\right\|^2 = \left\|  {\bf x}\right\|^2$. The new graph operator $\bA_\phi$ of an undirected graph has such property since $\bA$ is symmetric. We will further discuss the general graph Fourier transform and its energy-preserving in next subsection.

\emph{Property 2}: ${\bA}_e^N \bx = \bx$.

The first property is energy-preserving. The second property specific for $\bA_e$ is consistent with classic signal processing for an important phase shift property of the shift operator. 

\label{property1}
Property 1 is of great importance in frequency domain graph signals filtering. From \eqref{eq-p1a} and \eqref{eq-p1b}, it can be seen that if the modulus of the eigenvalues of the shift operator is not 1, the frequency components with small eigenvalues will disappear after several shifts. Indeed, only the frequency component with the largest eigenvalue will remain after many shifts. This is apparently undesirable. Property 1 implies that the new energy-preserving graph shift operator will preserve the energy of all frequency components. Moreover, as will be discussed later in this paper, linear shift invariant graph filter is defined by a polynomial of the graph shift operator, i.e., ${\bf H}=\sum_{k=0}^{L-1}h_k{{\bf A}_{e}}^k$, 
where the following properties hold true:

- Shift operator does not change the energy of the signal in the frequency domain. It only changes the phases of its frequency components

- Filter coefficients can modify the energy contents of the graph signal in the frequency domain.

We will see additional important property of ${\bA}_e$ in filtering and spectral analysis in later sections.

We note that our definition of the graph shift operator brings us the benefit to express the filtering operations in a more compact and meaningful form, similar to their counterparts in classical signal processing. For the choice of $\lambda_{e_k}=e^{-j \frac{2 \pi (k-1)}{N}}$, the shift operator ${\bf A}_e$ may not be sparse. Also ${\bf A}_e(i,j)$ may not be real-valued. Therefore, a large memory may be needed to save the corresponding operator and conduct the filtering operation in the shift domain. We note that such large memory may not be necessary if the filtering operation is conducted in the Fourier domain. We will also show that an LSI filter with a non-sparse shift operator may be represented by a polynomial of a sparse graph operator and thus has efficient shift domain implementation. \label{ExtraExplanationOfShift}

{\emph{Remark:}} Most of existing shift operators in the literature do not satisfy the energy-preserving property. For instance, in \cite{Moura1}, the graph shift operator is the adjacency matrix ${\bf A}$ of the graph. When such a shift operator applies to a graph signal, the energy content of the graph signal changes. To show this, note that ${\bf A}={\bf V}{\boldsymbol \Lambda}{\bf V}^{-1}$. Applying the graph shift operator $n$ times to the graph signal $\bf x$ results in ${\bf x}_n={\bf A}^n{\bf x}={\bf V}{\boldsymbol \Lambda}^n{\bf V}^{-1}{\bf x}$. Since the magnitude of the diagonal elements of ${\boldsymbol \Lambda}$ in general are not equal to 1, as $n$ becomes larger, some of the eigenvalues of ${\boldsymbol \Lambda}^n$ grow exponentially and the other eigenvalues approach zero. This means that the energy content of the signal is not preserved.

We also note that there exist other definitions of the graph shift operator in the literature such as $\frac{1}{\lambda_{\max}({\bf A})}{\bf A}$ as the \emph{normalized graph shift matrix} \cite{Moura2} where $\lambda_{\max}({\bf A})$ is the maximum eigenvalue of $\bf A$ or Laplacian matrix based shift operators \cite{Ortega,Segarra2016}.
Not only do these shift operators not preserve the energy, but also they actually filter the signals in that they modify the relative strength of different eigenvectors (frequency components). 
In \cite{Ortega}, the translation on graph is defined via generalized convolution with a delta centered at vertex $n$.  However, this translation operator aims to produce a geometrically localized shift in the vertex domain and does not preserve the energy. 
In \cite{GraphTranslation,GiraultThesis}, a new isometric shift operator has recently been introduced that satisfies the energy-preserving property with a similar general expression. It is indeed a special case of ${\bf A}_{\phi}$. Its eigenvalues are derived from the eigenvalue of the graph Laplacian matrix. Note that in our definition of graph shift operator, the eigenvalues (phase shifts) are detached from the eigenvalues of the graph adjacency matrix or Laplacian matrix and therefore are more flexible to accommodate other properties such as property 2 above.  We will further show that our new shift operators have properties leading to computationally efficient spectral analysis and filtering through the detailed formulations of graph Fourier analysis in the next subsection.

\subsection{Frequency Content of Graphs, Eigengraphs and Graph Fourier Basis}\label{sec-gft}

Consider the graph adjacency matrix ${\bf A}$ and its eigenvalue decomposition as ${\bf A}={\bf V}{\boldsymbol \Lambda}{\bf V}^{-1}$, where ${\boldsymbol \Lambda}$ is a diagonal matrix whose $i$-th diagonal element is the $i$-th eigenvalue of ${\bf A}$. Note that in this paper, we will always assume that $\bA$ is diagonalizable and eigenspaces have dimension equal to one for simplicity.

In graph theory, the eigenvalues of the graph adjacency matrix are called the spectrum of the graph \cite{GraphBook}. 

\subsubsection{Graph frequency content} \label{subsec-gfc}
Defining ${\bf V}=[{\bf v}_1 \; {\bf v}_2 \; \cdots \; {\bf v}_N]$ and $({\bf V}^{-1})^T=[\tilde{\bf v}_1 \; \tilde{\bf v}_2 \; \cdots \; \tilde{\bf v}_N]$, where ${\bf v}_i$ and $\tilde{\bf v}_i$ are $N \times 1$ column vectors of ${\bf V}$ and $({\bf V}^{-1})^T$, respectively, one can show that
\begin{align}\label{EigenGraph}
 {\bf A}={\bf V}{\boldsymbol \Lambda}{\bf V}^{-1}=\sum_{i=1}^N\lambda_i{\bf v}_i \tilde{\bf v}_i^T=\sum_{i=1}^N\lambda_i \hat{\bf V}_i.
\end{align}
 The rank one matrix $\hat{\bf V}_i={\bf v}_i \tilde{\bf v}_i^T$ is called the $i$-th \emph{eigengraph}, and ${\bf v}_i$ is the  $i$-th frequency component, of ${\bf A}$. Moreover, if none of the elements of ${\bf v}_i $ and $ \tilde{\bf v}_i$ are zero, the corresponding eigengraph is a complete graph, meaning that all nodes are connected to each other. However, the original graph that is a linear combination of the eigengraphs, stated in \eqref{EigenGraph}, may not be complete.

\emph{Remark}: The eigengraphs of the graph shift operator ${\bf A}_{\phi}$ are the same as those of the adjacency matrix ${\bf A}$ by definition. 

\subsubsection{Eigengraph structure}\label{sec-eigenstructure}

To elaborate more on eigengraph structures, let us define ${\bf v}_i\triangleq[v_{i1} \; v_{i2} \; \cdots \; v_{iN}]^T$ and $\tilde{\bf v}_i\triangleq[\tilde{v}_{i1} \; \tilde{v}_{i2} \; \cdots \; \tilde{v}_{iN}]^T$.
The corresponding $i$-th eigengraph is given by the rank one matrix
\begin{align}
 \hat{\bf V}_i={\bf v}_i \tilde{\bf v}_i^T=\begin{pmatrix} v_{i1}\tilde{v}_{i1}&v_{i1}\tilde{v}_{i2}& \cdots & v_{i1}\tilde{v}_{iN} \\ v_{i2}\tilde{v}_{i1}&v_{i2}\tilde{v}_{i2} & \cdots &v_{i2}\tilde{v}_{iN} \\ \vdots & \vdots & \ddots & \vdots \\
v_{iN}\tilde{v}_{i1} & v_{iN}\tilde{v}_{i2} & \cdots & v_{iN}\tilde{v}_{iN}
\end{pmatrix},
\end{align}
where $\hat{\bf V}_i(l,m)=v_{il}\tilde{v}_{im}$. The adjacency matrix of an eigengraph can be viewed as a signal/information transition matrix, where the weight $v_{il}\tilde{v}_{im}$ is the transition weight from node $l$ to node $m$. For instance, the eigengraph of a three node graph and the transition (bipartite) graph is shown in Fig.~\ref{fig:SimpleExample}(a) and Fig.~\ref{fig:SimpleExample}(b). A more general $N$ node eigengraph is shown in Fig.~\ref{fig:SimpleExample}(c). Note that for the $i$-th rank one eigengraph, the outgoing weight of node $l$ is $v_{il}$ and the incoming weight of node $m$ is $\tilde v_{im}$, see Fig.~\ref{fig:SimpleExample}(b). and Fig.~\ref{fig:SimpleExample}(c). We note that in these figures, $w_{lm}^i= {\bf v}_{im}\tilde{\bf v}_{il}$ is the signal transition weight from node $l$ to node $m$. \label{FigureExplanation} 

Note that an eigengraph $\hat{\bf V}_i$ is a special graph such that $\bA \hat{\bV}_i = \lambda_i \hat{\bf V}_i$. Although an eigengraph is generally a complete graph, a linear combination of the eigengraphs may not be complete, as is evident for the cyclic graph.

\subsubsection{Graph Fourier basis and Graph Fourier transform (GFT)}\label{sec-subGFT}

We refer to $\mathcal{F}={\bf V}^{-1}$ as the graph Fourier transform (GFT) operator since its rows, span a basis to represent the graph signal. The Fourier transform of a graph signal $\bx$ is $\bx_\mathcal{F}={\bf V}^{-1}\bx$. Thus $\mathcal{F}^{-1}={\bf V}$ is the inverse graph Fourier transform (IGFT) operator.

Note that the rows of $\mathcal{F}={\bf V}^{-1}$ are not orthogonal (unitary) for a general shape graph. However, one can easily verify that the vector space $ \text{Span}_k \{ \tilde{\bf v}_k\}$ of the columns of $({\bf V}^{-1})^T$ and the vector space $ \text{Span}_k \{ {\bf v}_k\} $ of  the columns of $\mathcal{F}^{-1}={\bf V}$ construct a biorthogonal basis, i.e., $\tilde{\bf v}_l^T {\bf v}_m=\delta_{l-m}$. We further note that the corresponding eigengraph of the $i$-th frequency component is constructed by a pair of $\tilde{\bf v}_i, \; {\bf v}_i$ meaning that they are constructed by the Fourier basis of the graph. This interpretation also confirms that the eigengraphs are the graph structures of graph frequency components, in which one can decompose a graph signal that is generated by the same graph structure, on those Fourier bases without any loss.

For biorthogonal GFT, we further define a dual GFT: $\mathcal{\tilde F}={\bf V}^{H}$ and an inverse dual GFT $\mathcal{\tilde F}^{-1}={\bf V}^{-H}$. As such, we have the inner product preservation: $\langle \bf x, \bf y \rangle = \langle {\bf \tilde{x}}_\mathcal{F}, {\bf y}_\mathcal{F} \rangle $, where ${\bf \tilde{x}}_\mathcal{F}= \mathcal{\tilde {F}}{\bf x} = {\bf V}^{H}{\bf x}$.

Note that such biorthogonal transform satisfies the frame theory \cite{FrameTheory}. More specifically, the energy of a graph signal in the Fourier domain, i.e., $\| {\bf x}_{\cal F} \|^2= ||{\bf V}^{-1}{\bf x}||_2^2$, is bounded by \begin{equation}\label{eq-frame1}
\alpha \|{\bf x}\|^2 \leq \| {\bf x}_{\cal F} \|^2 \leq \beta \|{\bf x}\|^2, 
\end{equation}
where $\alpha=\frac{1}{||{\bf V}^{-1}||_2^2}$ and $\beta=||{\bf V}^{-1}||_2^2$. In other words, 
\begin{equation}
\frac{1}{\beta} \| {\bf x}_{\cal F} \|^2 \leq \| {\bf x} \|^2 \leq \frac{1}{\alpha} \| {\bf x}_{\cal F}\|^2. 
\end{equation}
Furthermore, for unitary transform operator, i.e., ${\bf V}^{-1}={\bf V}^H$, $\alpha=\beta=1$. Therefore, such a transform preserves the energy in both shift and transform domains. The graph shift operator $\bA_\phi$ of an undirected graph has such property since its adjacency matrix $\bA$ is symmetric and the corresponding $\bV$ is unitary operator. Therefore $\bA_\phi$ for an undirected graph is a unitary operator by construction.

The shift operator $\bA_\phi$ is a linear combination of eigengraphs. It may not be a local (sparse) operator, meaning that most entries of this matrix may be non-zero. This means that the complexity of applying $\bA_\phi$ to a graph signal of size $N$ is of order of $O(N^2)$. However, once the signal is transformed to the Fourier domain, several other operations such as filtering are computationally efficient, as will be discussed later in section \ref{sec-graphwinerfilter}. Also, we will show in Theorem~\ref{TheoremReversePolynomial} that ${\bf A}_\phi$ can be represented as a polynomial of the adjacency matrix $\bA$ in certain condition and thus has efficient local implementation.

\emph{Remark:} 
For a) the adjacency matrix of the undirected graph, b) the combinatorial graph Laplacian, and c) the normalized Laplacian matrix, the GFT matrix V is unitary and the GFT becomes orthogonal.
\label{para-A-remark}


\subsubsection{Linear operator, projection operator and graph shift operator}\label{subsec-linearoperator} 
 We note that a linear operator on a graph signal can be defined as $\bf L$ such that if it applies to a graph signal ${\bf x}$, the result is also a graph signal $\bf y$ in which $\bf y=Lx$. A projection operator $\bf W$ satisfies
\begin{align}
	{\bf W}^k {\bf x}={\bf W} {\bf x}, \;\; \text{for all $k \in N$}.
\end{align}
It can be shown that, eigengraph operator, i.e., $\hat{\bf V}_i$, satisfy this property and thus a projection operator. The eigengraph operator $\hat{\bf V}_i$ represents the $i$-th basis for decomposition of the graph adjacency matrix $\bf A$. It means that, a graph structure, i.e., the graph adjacency matrix, is composed of a linear combination of $n$ independent eigengraphs (as we assume that all eigenvalues of $\bf A$ are distinct, thus eigenvectors are linearly independent). Moreover, applying an eigengraph operator to a graph signal will select the corresponding frequency component of the graph signal. This operation is in accordance with the classical signal processing interpretation of a \emph{filter operation}. More specifically, if a frequency selective filter applies to a signal several times, it returns the same frequency components of the signal similar to the case where the operator applies once. By defining the $\hat{\bf V}_i$ as the $i$-th frequency component of the graph, we interpret the $\lambda_i$ in ${\bf A}=\sum_{i=1}^{N}\lambda_i \hat{\bf V}_i$ as the significance of the corresponding frequency component. We further note that, indeed the frequency interpretation of time-series data comes from the linear cyclic graph structure of the time series data.
		
We emphasize that, a graph shift operator should preserve the frequency contents of a graph. Therefore, it should be an equally weighted linear combination of the eigengraphs with only phase shifts. In other words, ${\bf A}_\phi=\sum_{i=1}^{N}\alpha_i \hat{\bf V}_i$, where $|\alpha_i|=1$. This left us with the choice that $\alpha_i=e^{j\phi_i}$, where $0 \leq \phi_i < 2 \pi$. We further note that, to have a graph shift operator with independent eigengraph representation, we assume that $\phi_i \neq \phi_j$, for all $i \neq j$. This result is in accordance with the definition of graph shift operator as we defined earlier.

$\bf Example$ : Consider the directed cyclic graph with three node as shown in Fig.~\ref{fig:SimpleExample2}. The adjacency matrix ${\bf A}_{\rm cyclic}$ of this graph is given by 
\begin{align}\label{eq-DFT3} 
{{\bf A}_{\rm cyclic}}\!= \!\!\begin{pmatrix} 0&0 & 1 \\ 1&0 &0 \\ 0 & 1 & 0 \end{pmatrix}\!\!= \underset{{\bf V}}{\underbrace{\frac{1}{\sqrt{3}}\begin{pmatrix} 1&1 & 1 \\ 1&e^{j\frac{2\pi1}{3}} &e^{j\frac{2\pi2}{3}} \\ 1 & e^{j\frac{2\pi2}{3}} & e^{j\frac{2\pi4}{3}} \end{pmatrix}}} \cdot \;\;\;\;\;\;\;\;\;\;\;\; \nonumber \\  \underset{{\boldsymbol \Lambda}_{\rm cyclic}} {\underbrace{\begin{pmatrix} e^{-j\frac{2\pi 0}{3}}&0 & 0 \\ 0&e^{-j\frac{2\pi 1}{3}} &0 \\ 0 & 0 & e^{-j\frac{2\pi 2}{3}} \end{pmatrix}}} \cdot \underset{{\bf V}^{-1}}{\underbrace{\frac{1}{\sqrt{3}}\begin{pmatrix} 1&1 & 1 \\ 1&e^{-j\frac{2\pi1}{3}} &e^{-j\frac{2\pi2}{3}} \\ 1 & e^{-j\frac{2\pi2}{3}} & e^{-j\frac{2\pi4}{3}} \end{pmatrix}}}, 
\end{align}
where ${\bf V}^{-1}$ is the discrete Fourier (DFT) transform matrix. The eigengraphs and the signal transition (bipartite) graphs of the graph structure in Fig.~\ref{fig:SimpleExample2} are shown in Fig.~\ref{fig:SimpleExample}(a) and Fig.~\ref{fig:SimpleExample}(b)., where $v_{il}=e^{\frac{2 \pi (l-1)(i-1)}{N}} $, $\tilde{v}_{im}=e^{-\frac{2 \pi (m-1)(i-1)}{N}}$. Also note that the ${\boldsymbol \Lambda}_{\rm cyclic}$ is of the form of the special ${\boldsymbol \Lambda}_{e}$ defined in \eqref{LambdaEquation} and the adjacency matrix $\bA_{cyclic}$ is exactly $\bA_e$. \label{para-Acyclic}

\begin{figure}[H] 
\begin{center}
\scalebox{0.8} 
{
\begin{pspicture}(1,2)(16.437813,4.1)
\psdots[dotsize=0.2](3.9639063,2.3665628)
\psdots[dotsize=0.2](6.343906,2.306563)
\psdots[dotsize=0.2](8.463906,2.2865627)
\psline[linewidth=0.04cm,arrowsize=0.05291667cm 2.0,arrowlength=1.4,arrowinset=0.4]{->}(4.3239064,2.306563)(6.2239065,2.306563)
\psline[linewidth=0.04cm,arrowsize=0.05291667cm 2.0,arrowlength=1.4,arrowinset=0.4]{->}(6.583906,2.2865627)(8.323907,2.2865627)
\psarc[linewidth=0.04,arrowsize=0.05291667cm 2.0,arrowlength=1.4,arrowinset=0.4]{->}(6.193907,0.8165628){2.79}{38.181786}{141.58194}
\usefont{T1}{ptm}{m}{n}
\rput(4.898281,2.066563){1}
\usefont{T1}{ptm}{m}{n}
\rput(7.36375,2.046563){1}
\usefont{T1}{ptm}{m}{n}
\rput(6.06375,3.9265625){1}
\usefont{T1}{ptm}{m}{n}
\rput(3.2,2.37){node 1}
\usefont{T1}{ptm}{m}{n}
\rput(9.2,2.37){node 3}
\usefont{T1}{ptm}{m}{n}
\rput(6.2,2){node 2}
\end{pspicture} 
}
\caption{Directed cyclic graph with three nodes, $v_{il}=e^{\frac{2 \pi (l-1)(i-1)}{N}} $, $\tilde{v}_{im}=e^{-\frac{2 \pi (m-1)(i-1)}{N}}$.}
    \label{fig:SimpleExample2}
\end{center}
\end{figure}
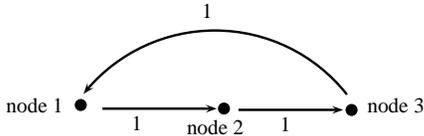

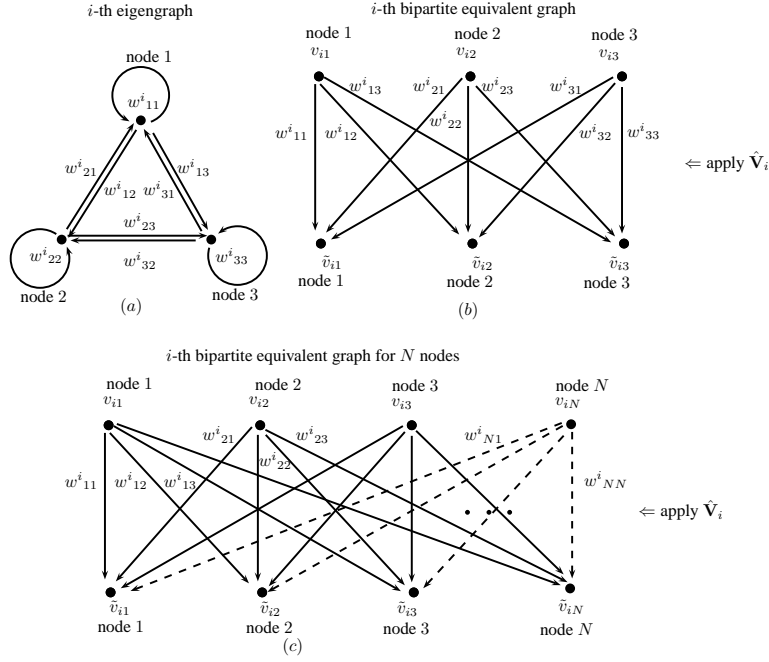
\begin{figure*}[htbp]
\begin{center}
\scalebox{0.65} 
{
\begin{pspicture}(0.46,-6.522969)(17.970469,6.522969)
\psdots[dotsize=0.20178913](6.6506248,4.976095)
\psdots[dotsize=0.20178913](9.750626,4.975532)
\psdots[dotsize=0.20178913](12.850625,4.975532)
\psdots[dotsize=0.20178913](6.6906247,1.5560949)
\psdots[dotsize=0.20178913](9.790626,1.555532)
\psdots[dotsize=0.20178913](12.890626,1.555532)
\usefont{T1}{ptm}{m}{n}
\usefont{T1}{ptm}{m}{n}
\psline[linewidth=0.04cm,arrowsize=0.05291667cm 2.0,arrowlength=1.4,arrowinset=0.4]{->}(6.574782,4.8195324)(6.574782,1.8195318)
\psline[linewidth=0.04cm,arrowsize=0.05291667cm 2.0,arrowlength=1.4,arrowinset=0.4]{->}(9.702514,4.7795324)(9.702514,1.779532)
\psline[linewidth=0.04cm,arrowsize=0.05291667cm 2.0,arrowlength=1.4,arrowinset=0.4]{->}(12.850425,4.7595325)(12.850425,1.759532)
\psline[linewidth=0.04cm,arrowsize=0.05291667cm 2.0,arrowlength=1.4,arrowinset=0.4]{->}(6.675677,4.8195324)(9.520904,1.739532)
\psline[linewidth=0.04cm,arrowsize=0.05291667cm 2.0,arrowlength=1.4,arrowinset=0.4]{->}(9.863945,4.7795324)(12.709174,1.6995322)
\psline[linewidth=0.04cm,arrowsize=0.05291667cm 2.0,arrowlength=1.4,arrowinset=0.4]{->}(12.709174,4.8395324)(9.924483,1.679532)
\psline[linewidth=0.04cm,arrowsize=0.05291667cm 2.0,arrowlength=1.4,arrowinset=0.4]{->}(9.581441,4.9395323)(6.796751,1.779532)
\psline[linewidth=0.04cm,arrowsize=0.05291667cm 2.0,arrowlength=1.4,arrowinset=0.4]{->}(6.756393,4.9595323)(12.628457,1.579532)
\psline[linewidth=0.04cm,arrowsize=0.05291667cm 2.0,arrowlength=1.4,arrowinset=0.4]{->}(12.668816,4.9595323)(6.897646,1.6395319)
\psdots[dotsize=0.2](3.0106254,4.0760946)
\psline[linewidth=0.04cm,arrowsize=0.05291667cm 2.0,arrowlength=1.4,arrowinset=0.4]{->}(1.5106254,1.7160946)(4.370625,1.7160946)
\psline[linewidth=0.04cm,arrowsize=0.05291667cm 2.0,arrowlength=1.4,arrowinset=0.4]{->}(4.2517967,1.8092037)(3.069454,3.9029856)
\psline[linewidth=0.04cm,arrowsize=0.05291667cm 2.0,arrowlength=1.4,arrowinset=0.4]{->}(2.915845,3.8815076)(1.5254059,1.6706816)
\psline[linewidth=0.04cm,arrowsize=0.05291667cm 2.0,arrowlength=1.4,arrowinset=0.4]{->}(3.2592785,3.8190951)(4.381972,1.8330942)
\psdots[dotsize=0.20178913](6.6506248,4.976095)
\psdots[dotsize=0.20178913](9.750626,4.975532)
\psdots[dotsize=0.20178913](12.850625,4.975532)
\psdots[dotsize=0.20178913](6.6906247,1.5560949)
\psdots[dotsize=0.20178913](9.790626,1.555532)
\psdots[dotsize=0.20178913](12.890626,1.555532)
\usefont{T1}{ptm}{m}{n}
\rput(6.691392,5.455531){$v_{i1}$}
\usefont{T1}{ptm}{m}{n}
\rput(6.9221077,1.175532){$\tilde{v}_{i1}$}
\usefont{T1}{ptm}{m}{n}
\rput(6.0923433,3.8395317){${w^i}_{11} $}
\usefont{T1}{ptm}{m}{n}
\rput(7.1109376,3.8395317){${w^i}_{12} $}
\usefont{T1}{ptm}{m}{n}
\rput(7.6109376,4.789532){${w^i}_{13} $}
\usefont{T1}{ptm}{m}{n}
\rput(8.910937,4.789532){${w^i}_{21} $}
\usefont{T1}{ptm}{m}{n}

\usefont{T1}{ptm}{m}{n}
\rput(10.310938,4.789532){${w^i}_{23} $}
\usefont{T1}{ptm}{m}{n}
\rput(9.250937,4.089532){${w^i}_{22} $}
\usefont{T1}{ptm}{m}{n}
\rput(11.7109375,4.789532){${w^i}_{31} $}
\usefont{T1}{ptm}{m}{n}
\rput(12.310938,3.8395317){${w^i}_{32} $}
\usefont{T1}{ptm}{m}{n}
\rput(13.292343,3.8595319){${w^i}_{33} $}
\usefont{T1}{ptm}{m}{n}
\rput(3.0723438,4.471703){${w^i}_{11} $}
\usefont{T1}{ptm}{m}{n}
\rput(1.0723437,1.3195323){${w^i}_{22} $}
\usefont{T1}{ptm}{m}{n}
\rput(4.8723435,1.2995323){${w^i}_{33} $}
\psline[linewidth=0.04cm,arrowsize=0.05291667cm 2.0,arrowlength=1.4,arrowinset=0.4]{->}(1.4938074,1.8500773)(2.867443,4.0221124)
\psline[linewidth=0.04cm,arrowsize=0.05291667cm 2.0,arrowlength=1.4,arrowinset=0.4]{->}(4.1308846,1.6187197)(1.5703664,1.6187197)
\usefont{T1}{ptm}{m}{n}
\rput(1.7909375,3.1195319){${w^i}_{21} $}
\usefont{T1}{ptm}{m}{n}
\rput(2.5909376,2.6595323){${w^i}_{12} $}
\usefont{T1}{ptm}{m}{n}
\rput(3.3909376,2.6595323){${w^i}_{31} $}
\usefont{T1}{ptm}{m}{n}
\rput(4.0909376,3.1095316){${w^i}_{13} $}
\usefont{T1}{ptm}{m}{n}
\rput(2.9909375,2.0095317){${w^i}_{23} $}
\usefont{T1}{ptm}{m}{n}
\rput(2.9909375,1.2095319){${w^i}_{32} $}
\usefont{T1}{ptm}{m}{n}
\rput(2.999063,6.2945313){$i$-th eigengraph}
\usefont{T1}{ptm}{m}{n}
\rput(9.814844,6.3345313){$i$-th bipartite equivalent graph}
\psarc[linewidth=0.04,arrowsize=0.05291667cm 2.0,arrowlength=1.4,arrowinset=0.4]{->}(3.0067189,4.5895324){0.58}{-68.9625}{246.8014}
\psarc[linewidth=0.04,arrowsize=0.05291667cm 2.0,arrowlength=1.4,arrowinset=0.4]{->}(4.9867187,1.3095319){0.6}{164.47589}{131.63354}
\psdots[dotsize=0.2](1.3906254,1.6360947)
\psdots[dotsize=0.2](4.4506254,1.6360947)
\psarc[linewidth=0.04,arrowsize=0.05291667cm 2.0,arrowlength=1.4,arrowinset=0.4]{->}(0.9467189,1.249532){0.6}{56.309933}{23.962488}
\psdots[dotsize=0.20178913](2.3506248,-2.163905)
\psdots[dotsize=0.20178913](5.450626,-2.1644678)
\psdots[dotsize=0.20178913](8.550625,-2.1644678)
\psdots[dotsize=0.20178913](2.3906248,-5.5839047)
\psdots[dotsize=0.20178913](5.490626,-5.584468)
\psdots[dotsize=0.20178913](8.590626,-5.584468)
\psline[linewidth=0.04cm,arrowsize=0.05291667cm 2.0,arrowlength=1.4,arrowinset=0.4]{->}(2.2747817,-2.320468)(2.2747817,-5.3204684)
\psline[linewidth=0.04cm,arrowsize=0.05291667cm 2.0,arrowlength=1.4,arrowinset=0.4]{->}(5.402514,-2.360468)(5.402514,-5.360468)
\psline[linewidth=0.04cm,arrowsize=0.05291667cm 2.0,arrowlength=1.4,arrowinset=0.4]{->}(8.550425,-2.380468)(8.550425,-5.380468)
\psline[linewidth=0.04cm,arrowsize=0.05291667cm 2.0,arrowlength=1.4,arrowinset=0.4]{->}(2.3756769,-2.320468)(5.220904,-5.400468)
\psline[linewidth=0.04cm,arrowsize=0.05291667cm 2.0,arrowlength=1.4,arrowinset=0.4]{->}(5.563945,-2.360468)(8.409174,-5.440468)
\psline[linewidth=0.04cm,arrowsize=0.05291667cm 2.0,arrowlength=1.4,arrowinset=0.4]{->}(8.409174,-2.300468)(5.6244826,-5.460468)
\psline[linewidth=0.04cm,arrowsize=0.05291667cm 2.0,arrowlength=1.4,arrowinset=0.4]{->}(5.281441,-2.200468)(2.4967508,-5.360468)
\psline[linewidth=0.04cm,arrowsize=0.05291667cm 2.0,arrowlength=1.4,arrowinset=0.4]{->}(2.4563928,-2.180468)(8.328457,-5.5604677)
\psline[linewidth=0.04cm,arrowsize=0.05291667cm 2.0,arrowlength=1.4,arrowinset=0.4]{->}(8.368815,-2.180468)(2.5976458,-5.500468)
\psdots[dotsize=0.20178913](2.3506248,-2.163905)
\psdots[dotsize=0.20178913](5.450626,-2.1644678)
\psdots[dotsize=0.20178913](8.550625,-2.1644678)
\psdots[dotsize=0.20178913](2.3906248,-5.5839047)
\psdots[dotsize=0.20178913](5.490626,-5.584468)
\psdots[dotsize=0.20178913](8.590626,-5.584468)
\usefont{T1}{ptm}{m}{n}
\usefont{T1}{ptm}{m}{n}
\usefont{T1}{ptm}{m}{n}
\rput(1.7923437,-3.3004682){${w^i}_{11} $}
\usefont{T1}{ptm}{m}{n}
\rput(2.8109376,-3.3004682){${w^i}_{12} $}
\usefont{T1}{ptm}{m}{n}
\rput(3.82,-3.3004682){${w^i}_{13} $}
\usefont{T1}{ptm}{m}{n}
\rput(4.6109376,-2.3504682){${w^i}_{21} $}
\usefont{T1}{ptm}{m}{n}
\rput(6.5200005,-2.3504682){${w^i}_{23} $}
\usefont{T1}{ptm}{m}{n}
\rput(5.76,-2.8794994){${w^i}_{22} $}
\usefont{T1}{ptm}{m}{n}
\rput(10.02,-2.3504682){${w^i}_{N1} $}
\usefont{T1}{ptm}{m}{n}
\rput(12.52,-3.280468){${w^i}_{NN} $}
\usefont{T1}{ptm}{m}{n}
\rput(6.529375,-0.76546824){$i$-th bipartite equivalent graph for $N$ nodes}
\psdots[dotsize=0.1](9.690626,-3.944468)
\psdots[dotsize=0.1](10.130626,-3.944468)
\psdots[dotsize=0.1](10.550627,-3.944468)
\psdots[dotsize=0.20178913](11.810627,-2.1444676)
\psdots[dotsize=0.20178913](11.790626,-5.504468)
\psline[linewidth=0.04cm,linestyle=dashed,dash=0.16cm 0.16cm,arrowsize=0.05291667cm 2.0,arrowlength=1.4,arrowinset=0.4]{->}(11.830425,-2.300468)(11.830425,-5.3004684)
\psline[linewidth=0.04cm,arrowsize=0.05291667cm 2.0,arrowlength=1.4,arrowinset=0.4]{->}(2.5185938,-2.1370308)(11.597188,-5.4370303)
\psline[linewidth=0.04cm,arrowsize=0.05291667cm 2.0,arrowlength=1.4,arrowinset=0.4]{->}(5.5571876,-2.2570307)(11.737187,-5.3770304)
\psline[linewidth=0.04cm,arrowsize=0.05291667cm 2.0,arrowlength=1.4,arrowinset=0.4]{->}(8.638594,-2.277031)(11.657187,-5.2170305)
\psline[linewidth=0.04cm,linestyle=dashed,dash=0.16cm 0.16cm,arrowsize=0.05291667cm 2.0,arrowlength=1.4,arrowinset=0.4]{->}(11.657187,-2.3770308)(8.784483,-5.500468)
\psline[linewidth=0.04cm,linestyle=dashed,dash=0.16cm 0.16cm,arrowsize=0.05291667cm 2.0,arrowlength=1.4,arrowinset=0.4]{->}(11.6171875,-2.1170309)(2.7971876,-5.5570307)
\psline[linewidth=0.04cm,linestyle=dashed,dash=0.16cm 0.16cm,arrowsize=0.05291667cm 2.0,arrowlength=1.4,arrowinset=0.4]{->}(11.637187,-2.1770306)(5.597188,-5.5770307)
\usefont{T1}{ptm}{m}{n}
\rput(2.8164976,0.25609434){$(a)$}
\usefont{T1}{ptm}{m}{n}
\rput(9.706497,0.17609432){$(b)$}
\usefont{T1}{ptm}{m}{n}
\rput(6.1164975,-6.7){$(c)$}
\usefont{T1}{ptm}{m}{n}
\rput(15.0,3.2395322){$\Leftarrow \text{apply } \hat{\bf V}_i$}
\usefont{T1}{ptm}{m}{n}
\rput(9.671392,5.455531){$v_{i2}$}
\usefont{T1}{ptm}{m}{n}
\rput(12.5913925,5.4155316){$v_{i3}$}
\usefont{T1}{ptm}{m}{n}
\rput(10.002108,1.1555319){$\tilde{v}_{i2}$}
\usefont{T1}{ptm}{m}{n}
\rput(12.802108,1.1555319){$\tilde{v}_{i3}$}
\usefont{T1}{ptm}{m}{n}
\rput(2.451392,-1.6844685){$v_{i1}$}
\usefont{T1}{ptm}{m}{n}
\rput(2.5821075,-5.904468){$\tilde{v}_{i1}$}
\usefont{T1}{ptm}{m}{n}
\rput(5.431392,-1.6844685){$v_{i2}$}
\usefont{T1}{ptm}{m}{n}
\rput(5.662108,-5.924468){$\tilde{v}_{i2}$}
\usefont{T1}{ptm}{m}{n}
\rput(8.351392,-1.7244685){$v_{i3}$}
\usefont{T1}{ptm}{m}{n}
\rput(8.462108,-5.924468){$\tilde{v}_{i3}$}
\usefont{T1}{ptm}{m}{n}
\rput(11.751392,-1.7044685){$v_{iN}$}
\usefont{T1}{ptm}{m}{n}
\rput(11.822108,-5.884468){$\tilde{v}_{iN}$}
\usefont{T1}{ptm}{m}{n}
\rput(3.1795318,5.3945312){node $1$}
\usefont{T1}{ptm}{m}{n}
\rput(1.0195318,0.47453126){node $2$}
\usefont{T1}{ptm}{m}{n}
\rput(4.939532,0.51453125){node $3$}
\usefont{T1}{ptm}{m}{n}
\rput(6.859532,5.8745313){node $1$}
\usefont{T1}{ptm}{m}{n}
\rput(9.899532,5.8545313){node $2$}
\usefont{T1}{ptm}{m}{n}
\rput(12.699532,5.8345313){node $3$}
\usefont{T1}{ptm}{m}{n}
\rput(12.039532,-1.3454688){node $N$}
\usefont{T1}{ptm}{m}{n}
\rput(6.699532,0.7945312){node $1$}
\usefont{T1}{ptm}{m}{n}
\rput(9.7395315,0.77453125){node $2$}
\usefont{T1}{ptm}{m}{n}
\rput(12.539532,0.75453126){node $3$}
\usefont{T1}{ptm}{m}{n}
\rput(2.779532,-1.3054688){node $1$}
\usefont{T1}{ptm}{m}{n}
\rput(5.819532,-1.3254688){node $2$}
\usefont{T1}{ptm}{m}{n}
\rput(8.619532,-1.3454688){node $3$}
\usefont{T1}{ptm}{m}{n}
\rput(2.599532,-6.2854686){node $1$}
\usefont{T1}{ptm}{m}{n}
\rput(5.6395316,-6.3054686){node $2$}
\usefont{T1}{ptm}{m}{n}
\rput(8.439532,-6.3254685){node $3$}
\usefont{T1}{ptm}{m}{n}
\rput(11.719532,-6.3454685){node $N$}
\usefont{T1}{ptm}{m}{n}
\rput(14.06,-3.9004679){$\Leftarrow \text{apply } \hat{\bf V}_i$}
\end{pspicture} 
}
\caption{Eigengraph structure for a three node graph is shown in $(a)$, where $w_{lm}^{i}=v_{im}\tilde{v}_{il} $ is the signal transition weight from node $l$ to node $m$. The signal/information transition representation of the $i$-th eigengraph is shown in $(b)$. In this figure, $v_{il}$ is the outgoing weight from node $l$ and $\tilde{v}_{im}$ is the incoming weight to node $m$. Such a weight distribution preserves the rank one property of the $i$-th eigengraph. Part $(c)$ is the generalization of the signal transition representation of the $i$-th eigengraph for an $N$ node graph structure.}
    \label{fig:SimpleExample}
\end{center}
\end{figure*}

\section{Graph Filters based on the New Shift Operator}\label{sec-gfgso}
In classical signal processing, filters are referred to operators that apply on a signal as input, and produce another signal as output. Filters can be categorized into different classes, e.g., continues time or discrete time, linear and nonlinear, time invariant and time-varying. The compatible category of filter classes to graph signals is discrete time linear filters. Linear filtering on graphs is represented by multiplying the input signal vector $\bf x$ by a matrix ${\bf H}\in \mathbb{C}^{N \times N}$, called filter matrix. The filtered output signal vector $\by ={\bf H}\bf x$. This filter operates on graph signals similar to the shift operator, i.e., the filtered signal $\by$ at the $i$-th vertex is a linear combination of the value of the original signal ${\bf x}$. More specifically, $\by(i)=\sum_{j=1}^{N}{\bf H}(i,j) {\bf x}(j)$.

\subsection{Linear Shift Invariant Graph Filters} \label{sec-LSI}

If we consider the shift operator on a graph to be ${\bf A}_{\phi}$, then the linear shift invariant property (LSI) of filters is ${\bf H}\bf{A}_{\phi} \bf x={\bf A}_{\phi}{\bf H} \bf x$. Indeed, this property implies that, the filter and the shift operator are commutable. It is straightforward to show that the following theorem in \cite{Moura1,Moura2} still hold for a graph LSI filter defined by shift operator $\bA_\phi$.
\begin{athm}\label{theorem-LSI}
Every polynomial of a square matrix ${{\bf A}_{\phi}}$ is a graph LSI filter and every graph LSI filter is a polynomial of a square matrix ${{\bf A}_{\phi}}$.
\end{athm}

This theorem shows that every LSI graph filter is a polynomial in the graph shift matrix, i.e., 
\begin{align}\label{eq-LSIFilter}
{\bf H}=h({{\bf A}_{\phi}})=\sum_{k=0}^{L-1}h_k{{\bf A}_{\phi}}^k
\end{align}
where $h_k$ is called the $k$-th tap of the graph filter and $(L-1)$ is the order of the polynomial representation of the LSI filter.

We therefore can prove the following theorem:
\begin{athm}\label{theorem-LSI-adjacent}
Any arbitrary adjacency matrix ${\bf A}$ is an LSI filter under ${\bf A}_\phi$. 
\end{athm}
\begin{proof}: We know from the definition of the graph shift operator that ${\bf A}={\bf A}_{h}{\bf A}_{\phi}={\bf A}_{\phi}{\bf A}_{h}$, and hence the LSI filter ${\bf H}={\bf A}_{h}$.
\end{proof}
\emph{Remark}: Indeed, one can write ${\bf A}=\sum_{k=0}^{L-1}h_k{{\bf A}_{\phi}}^k$. We note that when the shift matrix is ${\bf A}_e$, one can show that $h_k=\sum_{l=1}^Ne^{j\frac{2 \pi kl}{N}}{\lambda_l}$, i.e., $h_k$ is the $k$-th coefficient of the IDFT of the eigenvalue vector ${\boldsymbol \lambda}$, where $\lambda_l$ is the $l$-th eigenvalue of $\bf A$. This result allows us to compute the filter coefficients more efficiently. Theorem~\ref{theorem-LSI-adjacent} shows that the adjacency matrix $\bA$ can be decomposed into two parts. The first part is the energy-preserving graph shift operator $\bA_\phi$, i.e., frequency components of a graph. The second part is the filtering part represented by eigenvalues of the adjacency matrix, which changes the amplitudes of the frequency components. 

\begin{athm}\label{TheoremReversePolynomial}
	The graph shift operator ${\bf A}_\phi$ can be written as a polynomial of the graph adjacency matrix ${\bf A}$, if the eigenvalues of ${\bf A}$ are all distinct.
\end{athm}
\begin{proof}
	See Appendix~\ref{ReversePolynomial}.
\end{proof}
\emph{Remark}: The adjacency matrix $\bA$ is often sparse (local). We can design an LSI filter using $\bA_\phi$ since it has good mathematical properties. We can then convert the LSI filter as a polynomial of $\bA$ such that it has an efficient graph domain implementations. Note that using $\bA_\phi$ can have efficient in graph frequency domain as we will discuss in section~\ref{sec-optimalfilter}.

$\bf Example$: Let us consider the discrete time circular convolution $y[n]=h[n]\circledast x[n]=\sum_{m \in N}x[m]h[n-m]$ in classical signal processing for periodic time series data. Such an operator can be cast into the matrix form ${\bf y}={\bf H}{\bf x}$ where ${\bf x}$ and ${\bf y}$ are the input and output signal vectors, respectively. The filter matrix ${\bf H}$ has the following Toeplitz form
\begin{equation}\label{eq-toeplitz}
{\bf H} = 
 \begin{pmatrix}
  h[0] & h[N-1] & h[N-2] & \cdots & h[1] \\
  h[1] & h[0]   &  h[N-1]& \cdots & h[2] \\
  \vdots  & \vdots  & \vdots & \ddots & \vdots  \\
  h[N-1] &  h[N-2]&\cdots & h[1] & h[0] 
 \end{pmatrix}.
\end{equation}
 Note that, with a small abuse of notation, we will use $h_i$ in the graph representation instead of $h[i]$ in the classical signal processing counterpart. Using some matrix calculation, one can show that the filter matrix can be written as a polynomial of the circulant adjacency matrix \eqref{Circulant} as ${\bf H}=h({\bf C})=\sum_{k=0}^{N-1}h_k{\bf{C}}^k$. Note that in the cyclic graph, the adjacency matrix is exactly the $\bA_e$ defined in \eqref{eq-Ae}, $\bA_e=\bC$. This means that the circular convolution is equivalent to the LSI graph filtering based on the graph representation of the periodic time series data. We note that every Toeplitz graph filter matrix can be considered as a linear time invariant filter for time series periodic data.

$\bf Example$: Let us now consider filtering aperiodic time series data in classical signal processing. We show that such a filtering operation is also equivalent to the LSI graph filtering. To show this, let us start with the traditional signal processing filtering as $y[n]=h[n]* x[n]=\sum_{m=-\infty}^{\infty}x[m]h[n-m]$. Without loss of generality, we assume that $x[n]\neq0$, for $0\leq n \leq N-1$ and $h[n]\neq0$, for $0\leq n \leq L-1$, and $L<N$. Defining ${\bf 0}_{1 \times L-1}\triangleq[0\;\;0\;\;\cdots 0]$, ${\bf x}\triangleq[x[0] 
\;\;x[1] \cdots x[N-1] 
\;
\;{\bf 0}_{1 \times L-1} ]^T$ and ${\bf y}\triangleq[y[0]\;\;y[1]\;\;\cdots y[N+L-1]]^T$, one can rewrite the filtering equation as ${\bf y}={\bf H}{\bf x}$ where ${\bf H}_{(N+L-1)\times (N+L-1)}$ is defined in \eqref{FilterStructure1}.
\begin{align} \label{FilterStructure1}
&{\bf H}_{(N+L-1)\times (N+L-1)}= \nonumber\\ 
&\begin{pmatrix}
h[0] & 0& 0 &\cdots&0&\cdots&0\\
h[1] & h[0]&0&\cdots &0 &\cdots&0\\
h[2]& h[1]& h[0]& \cdots&0& \cdots&0\\
\vdots&\vdots&\vdots&\vdots&\cdots&\ddots&\vdots\\
0& 0 &\cdots&\cdots&h[L-1]& \cdots&h[0]\\
\end{pmatrix}.
\end{align}
The output of such a filtering operation, i.e., ${\bf y}={\bf H}{\bf x}$, is equivalent to that of the ${\bf y}=\tilde{\bf H}{\bf x}$, where $\tilde{\bf H}=\sum_{l=0}^{L-1}h_l{\bf C}^l$. Note that, in this example  ${\bf C}$ is the $(N+L-1)\times (N+L-1)$ circulant matrix defined in \eqref{Circulant}. One can easily show that $\tilde{\bf H}$ can be written as \eqref{eq-Haperiodic}.
\begin{figure*}[htbp]
\begin{align}\label{eq-Haperiodic}
\tilde{\bf H}_{(N+L-1)\!\times (N+L-1)}\!\!=\!\!\begin{pmatrix}
h[0] & 0& 0 &\cdots&0&h[L-1]&\cdots&h[2]&h[1]\\
h[1] & h[0]&0&\cdots &0 &0&h[L-1]&\cdots&h[2]\\
h[2]& h[1]& h[0]& \cdots&0& 0&0 & \cdots& h[3]\\
\vdots&\vdots&\vdots&\vdots&\vdots&\vdots&\vdots&\ddots&\vdots\\
0& 0 &\cdots&\cdots&h[L-1]& h[L-2]&h[L-3]&\cdots &h[0]\\
\end{pmatrix}
\end{align}
\end{figure*}
This result shows that filtering aperiodic discrete time signals by the filter operator ${\bf H }$ is equivalent to filtering the zero-padded graph signal $\bf x$ by the graph filter $\tilde{\bf H}$. Note that convolution for aperiodic signals is equivalent to circular convolution of the zero-padded periodic versions of the input signal, and hence it is equivalent to graph filtering where the graph filter is defined by the filter matrix $\tilde{\bf H}$. More specifically, let us consider the zero-padded periodic versions of the aperiodic signal $x[n] $ and filter $h[n]$ as ${\hat x}[n]={\hat x}[n+N+L]$ and ${\hat h}[n]={\hat h}[n+N+L]$ where
\begin{align}
&{\hat x}[n]\!\!=\!\!\begin{cases}x[n] & \!\!n=0,1,\cdots N-1 \\
0 &\!\! n=N,N+1,\cdots,N+L-1
\end{cases}\nonumber \\ &{\hat h}[n]\!\!=\!\!\begin{cases}h[n] &\!\! n=0,1,\cdots L-1 \\
0 &\!\! n=L,L+1,\cdots,N+L-1
\end{cases}
\end{align}
then, $x[n]*h[n]=\hat{x}[n]\circledast\hat{h}[n]$, for all $n$. Hence, filtering aperiodic time series data can be written as the graph LSI filtering of the zero-padded graph signals.

\begin{athm}\label{T3}
 When $L_{{\bf A}_{\phi}}\leq L$, for a LSI filter  there exists an equivalent form for the LSI filter in \eqref{eq-LSIFilter} as ${\bf H}=\breve{h}({{\bf A}_{\phi}})=\sum_{k=0}^{L_{{\bf A}_{\phi}}-1}\breve{h}_k{\bf A}_{\phi}^k$, where $L_{{\bf A}_{\phi}}$ is the degree of the minimal polynomial of  ${\bf A}_{\phi}$. Moreover, there exist a closed-form expression for the filter taps $\breve{h}_k$ as a function of $h_k$.
\end{athm}
\begin{proof} See Appendix~\ref{Theorem3} for a constructive proof.  
\end{proof}
Finite/infinite impulse response (FIR/IIR) filters are certain types of filters with great importance in classical signal processing and have simple frequency domain interpretation. We herein aim to bring those concepts to the graph signal processing as GFIR and GIIR filters where G stands for graph representation. As we have shown in Theorem~\ref{T3}, any LSI filter can be written as a polynomial of the graph shift operator with the maximum order of $L_{{\bf A}_{\phi}}-1$. We therefore can define the GFIR and GIIR filters. \label{FIR/IIR Definition}

${\bf Definition}$ : We define a GFIR filter to be ${\bf H}=\sum_{k=0}^{L-1}h_k{{\bf A}_{\phi}}^k$ where $L<L_{{\bf A}_{\phi}}$ and a GIIR filter to be ${\bf H}=\sum_{k=0}^{L-1}h_k{{\bf A}_{\phi}}^k$ where $L=L_{{\bf A}_{\phi}}$.

\subsection{Frequency domain interpretation of filtering} \label{sec-freqfilter}

The derivation of the results presented earlier in this section is in the time domain for time series data or the shift domain for graph signals. One can also describe the filtering process, equivalently, in the frequency domain obtained by the Fourier transform operator. More specifically, if ${\bf y}={\bf H}{\bf x}$ is the filtering operation in the time/shift domain, it can also be represented in the frequency domain as ${\bf y}_\mathcal{F}={\bf H}_\mathcal{F}{\bf x}_\mathcal{F}$, where the subscript $\mathcal{F}$ stands for \emph{Fourier} transformed versions of the corresponding signals/filters. Note that we have used ${\bf y}_\mathcal{F}={\bf V}^{-1}{\bf y}$, ${\bf x}_\mathcal{F}={\bf V}^{-1}{\bf x}$ and ${\bf H}_\mathcal{F}={\bf V}^{-1}{\bf H}{\bf V}$. Note that the filtering process ${\bf H}{\bf x}$ (matrix and vector multiplication) in the shift domain has a simpler representation in the Fourier domain as suggested by ${\bf H}_\mathcal{F}{\bf x}_\mathcal{F}$. To show this, we note that
\begin{align}\label{FrequencyDomainFilter}
{\bf H}_\mathcal{F}&=\!\!\!\!\sum_{k=0}^{L-1}h_k{{\boldsymbol \Lambda}_{\phi}}^k\nonumber \\ &\!\!\!\!\!\!\!\!\!=\!\!\!\begin{pmatrix}
  \sum_{k=0}^{L-1}h_k({{ \lambda}_{\phi_1}})^k &\!\!\!\!\!\!\!\!\!\!\!\!0 &\!\!\!\!\!\!\! \cdots &\!\!\!\!\!\!\!0 \\
  0    &\!\!\!\!\!\!\!\!\!\!\!\!  \sum_{k=0}^{L-1}h_k({{ \lambda}_{\phi_2}})^k&\!\!\!\!\!\!\!\! \cdots &\!\!\!\!\!\!\!\! 0 \\
  \vdots   &\!\!\!\!\!\!\!\!\!\!\!\! \vdots &\!\!\!\!\!\!\!\! \ddots &\!\!\!\!\!\!\!\! \vdots  \\
  0 &\!\!\!\!\!\!\!\!\!\!\!\!0 &\!\!\!\!\!\!\! \cdots &\!\!\!\!\!\!\!\! \sum_{k=0}^{L-1}h_k({{ \lambda}_{\phi_N}})^k 
 \!\!\end{pmatrix}
\end{align}
Therefore, the filtering ${\bf H}_\mathcal{F}{\bf x}_\mathcal{F}$ in the Fourier domain is a simple point-wise multiplication. More specifically, ${\bf y}(m)=\sum_{n=1}^N{\bf H}(m,n){\bf x}(n)$. However,  ${\bf y}_\mathcal{F}(m)={\bf H}_\mathcal{F}(m,m)\times{\bf x}_\mathcal{F}(m)$.

Note that for $\bA_\phi$,
\begin{align}\label{eq-HDTFT}
{\bf H}_\mathcal{F}(m,m) &= \sum_{k=0}^{L-1}h_k({{ \lambda}_{\phi_m}})^k 
= \sum_{k=0}^{L-1}h_k({e^{j\phi_m k}}) \nonumber\\
&= \bH_{DTFT}(\omega|\omega = \phi_m),
\end{align}
and for $\bA_e$, $\phi_m = -\frac{2 \pi m}{N}$, and thus
\begin{align}\label{eq-HDTFT-Ae}
{\bf H}_\mathcal{F}(m,m) &= \sum_{k=0}^{L-1}h_k({{ \lambda}_{e_m}})^k  \nonumber\\
&= \bH_{DTFT}\left(\omega\big |\omega = -\frac{2 \pi m}{N}\right).
\end{align}

Interestingly, with the new set of energy-preserving shift operators $\bA_e$, the GFT coefficient of a LSI filter $\bH$ can be computed using $L$-tap discrete-time Fourier transforms (DTFT). 

\label{filtering}
We further note that, the filtering equation \eqref{FrequencyDomainFilter} is composed of two components; the filter coefficients $h_k$'s and the complex exponentials  $\lambda_{\phi_i}$ that are the eigenvalues of our defined graph shift operator. Since the magnitude of $\lambda_{\phi_i}$'s are one,  $\lambda_{\phi_i}^k$ does not change the magnitude of the $k$-th component of the graph filter ${\bf H}_\mathcal{F}(i,i)$, i.e.,  $h_k (\lambda_{\phi_i})^k$. This suggests that the filter coefficients $h_k$'s are the source of the change of the $k$-th component of ${\bf H}_\mathcal{F}(i,i)$. In contrast, the magnitude of the eigenvalues of arbitrary graph shift operators proposed in the literature are not normalized to one. Thus, the role of the filter coefficients (controlling the signal level/energy) are diminished. It leads to a saturating filtering performance, which is illustrated in the simulation section.

\section{Correlation Functions of Graph Signals and Optimal LSI Graph Filters}\label{sec-optimalfilter}

In this section, we assume that the structure of the graph is known, meaning that the graph adjacency matrix ${\bf A}$ and the related shift operator ${\bf A}_{\phi}$ is given. Assuming that ${\bf A}_{\phi}$ is known, we aim to obtain the structure of the graph LSI filters such that a certain set of constraints are satisfied. We discuss several filter design problems in GSP in the sequel that arise in classical signal processing. 

\subsection{Wiener Filter for Directed Cyclic Graph Data (Time Series)}\label{sec-wienerfilter}
We will first reformulate the time series signal Wiener filter using GSP representation and then generalize it to arbitrary graph signals. Consider the graph representation of the time series data in Fig.~\ref{fig:timeseries}. Assume that $\bf x$ is the graph signal and $\bf y$ is a noisy measurement of the graph signal $\bf x$: $y_i = x_i+ n_i$, where $n_i$ is i.i.d. zero mean white Gaussian noise. A conventional question in denoising problems is to design an LSI filter such that the residual error $\|{\bf H}\bf y-\bf x\|_2^2$ is minimum. Strictly speaking, when $\bf x$ and $\bf y$ are given, we aim to solve the following optimization problem
\begin{align}\label{Opt1}
\min_{\bf H} \;\|{\bf H}\bf y-\bf x\|_2^2.
\end{align}
Since ${\bf H}$ is shift invariant, it can be written as ${\bf H}=\sum_{k=0}^{L-1}h_l{\bf C}^l$, where ${\bf C}$ is defined earlier. Note that the filtered signal can be rewritten as ${\bf H}{\bf y}=\sum_{k=0}^{L-1}h_l{\bf C}^l{\bf y}={\bf B}{\bf h}$ where
\begin{align}\label{BMatirx}
{\bf B}_{N \times L} =[{\bf y} \;\;{\bf C}^1{\bf y} \cdots {\bf C}^{L-1}{\bf y}], \; \;{\bf h}=[h_0\;\; h_1 \cdots h_{L-1}]^T.
\end{align}
Rewrite~\eqref{Opt1} by replacing ${\bf H}{\bf y}$ by its equivalent ${\bf B}{\bf h}$,  
\begin{align}\label{Opt11}
\min_{\bf h} \;\|{\bf B}{\bf h}-\bf x\|_2^2.
\end{align}
Since $L=L_{{\bf A}_{\phi}}\leq N$ is the degree of minimal polynomial of the graph shift matrix, we only consider the cases where $L \leq N$ (Note that when $N=L$, the matrix ${\bf B}$ is full rank, hence the solution to the optimization problem \eqref{Opt11} can be written as ${\bf h}^{\rm o}={\bf B}^{-1} {\bf x}$). If $L<N$, the solution of least square optimization problem \eqref{Opt11} can be obtained by solving 
\begin{align}\label{SolutionWeiner}
{\bf B}^H{\bf B}{\bf h}={\bf B}^H{\bf x},
\end{align}
where $H$ is the Hermitian operator. Such a solution has an interesting interpretation for time series data as will be shown in the sequel. We note that \eqref{SolutionWeiner} can be written as
\begin{align}
&\begin{pmatrix}
{\bf y}^H \\{\bf y}^H{\bf C}^H \\\vdots \\{\bf y}^H({\bf C}^{L-1})^H
\end{pmatrix}[{\bf y} \;\;{\bf C}{\bf y} \cdots {\bf C}^{L-1}{\bf y}]{\bf h}=\begin{pmatrix}
{\bf y}^H \\{\bf y}^H{\bf C}^H \\\vdots \\{\bf y}^H({\bf C}^{L-1})^H
\end{pmatrix}{\bf x} \nonumber
\end{align}
or, equivalently, as \eqref{SimplifiedSolution}.
\begin{figure*}
\begin{align}\label{SimplifiedSolution}
&\begin{pmatrix}
{\bf y}^H{\bf y}&& {\bf y}^H{\bf C}{\bf y}&&\!\!\! \cdots\!\!\! &&{\bf y}^H({\bf C})^{L-1}{\bf y} \\{\bf y}^H{\bf C}^H{\bf y} && {\bf y}^H{\bf C}^H{\bf C}{\bf y} && \!\!\!\cdots \!\!\!&&{\bf y}^H{\bf C}^H({\bf C})^{L-1}{\bf y}  \\ \vdots&& \vdots && \!\!\!\ddots \!\!\!&&\vdots \\{\bf y}^H({\bf C}^{L-1})^H{\bf y} &&{\bf y}^H({\bf C}^{L-1})^H{\bf C}{\bf y}&&\!\!\! \cdots \!\!\!&&{\bf y}^H({\bf C}^{L-1})^H({\bf C}^{L-1}){\bf y}
\end{pmatrix}\!\!{\bf h}\!=\!\!\!\begin{pmatrix}
{\bf y}^H{\bf x} \\{\bf y}^H{\bf C}^H{\bf x} \\\vdots \\{\bf y}^H({\bf C}^{L-1})^H{\bf x}
\end{pmatrix}.
\end{align}
\end{figure*}
We note that the circulant matrix ${\bf C}$ has the unitary property, i.e., $({\bf C}^H)^k{\bf C}^k={\bf I}$, $\forall k$. Moreover, we claim that ${\bf y}^H({\bf C}^l)^H{\bf y}$ is the autocorrelation of the vector $\bf y$ at lag $l$. To show this, we first note that ${\bf C}^l{\bf y}$ is the circularly shifted version of the $\bf y$ by amount $l$. Defining the autocorrelation function of $\bf y$ as 
\begin{equation}\label{eq-autocorr}
R_{\bf yy}(l)\triangleq \sum_{n} y_ny_{n+l}^*,
\end{equation}
one can easily show that 
\begin{align}\label{eq-autocorr1}
{\by}^H({\bC}^l)^H {\bC}^m{\by}&=({\bf C}^{l-m}{\bf y})^H{\bf y} \nonumber\\
&=\sum_{n} y_ny_{n+l-m}^*=R_{\bf yy}(l-m),
\end{align}
where $y_k$ is the $(k\;{\text{mod}\;N})$-th element of the vector $\bf y$ and $^*$ is the conjugation operator. We also define the cross-correlation between the input and output vectors $\bf x$ and $\bf y$, as 
\begin{equation}\label{eq-crosscorr}
r_{\bf xy}(l)\triangleq \sum_{n\in N} x_ny_{n+l}^*={\bf y}^H({\bf C}^l)^H{\bf x}.
\end{equation} 
The linear equations \eqref{SimplifiedSolution} can hence be rewritten as
\begin{equation}\label{eq-WienerHopf1}
{\bf B}^H {\bf Bh} = {\bf R}{\bf _{yy}}{\bf h}={\bf r}{\bf _{xy}},
\end{equation}
i.e.,
\begin{align}\label{eq-WienerHopfMat}
\!\!\!\!\!\!\!\!\begin{pmatrix}
\!\!\!\!R_{\bf yy}(0)&&\!\!\!\!\!\!\!\!\!\! R_{\bf yy}^*(1)\!\!&&\!\!\!\!\!\!\!\! \cdots &&\!\!\!\!\!\!\!\!\!\!R_{\bf yy}^*(L\!\!-\!\!1) \\ \!\!\!\!R_{\bf yy}(1) &&\!\!\!\!\!\! \!\!\!\!R_{\bf yy}(0)\!\! &&\!\!\!\!\!\!\!\! \cdots &&\!\!\!\!\!\!\!\!\!\!R_{\bf yy}^*(L\!\!-\!\!2) \\ \!\!\!\! \vdots&&\!\!\!\!\!\!\!\!\!\! \vdots\!\! &&\!\!\!\!\!\!\!\! \ddots &&\!\!\!\!\!\!\!\! \!\!\vdots \\ \! R_{\bf yy}(L\!\!-\!\!1) &&\!\!\!\!\!\!\!\!\!\! R_{\bf yy}(L\!\!-\!\!2) \!\!&&\!\!\!\!\!\!\!\! \cdots&&\!\!\!\!\!\! \!\!\!\!R_{\bf yy}(0)
\end{pmatrix}\!\!{\bf h}\!\!=\!\!\begin{pmatrix}
\!\!\!\!r_{\bf xy}(0) \\\!\!\!\!r_{\bf xy}(1) \\\!\!\!\! \vdots \\\!r_{\bf xy}(L\!\!-\!\!1)
\end{pmatrix}.
\end{align}
Eq.~\eqref{eq-WienerHopfMat} is indeed the Wiener-Hopf equation. Note that $\bR_{\by\by}$ is a Toeplitz matrix.

Note that the LSI property of graph filters for time series data leads to the Wiener filtering in the classic signal processing. One can also compute the autocorrelation and cross-correlation more efficiently as
\begin{align} 
{\bf R}{\bf _{yy}}(i,j) &= R_{\bf yy}(i-j) \nonumber \\
	&= {\bf y}_\mathcal{F}^{H}({\bf \Lambda}^*)^{i-1}{\bf \Lambda}^{j-1}{\bf y}_\mathcal{F} = \sum_{n=1}^{N}|{\bf y}_\mathcal{F}(n)|^2\lambda_n^{j-i}, \label{eq-Ryy} 
	\\
{\bf r}{\bf _{xy}}(i)&={\bf y}_\mathcal{F}^{H}{({\bf \Lambda}^*)}^{i-1}{\bf x}_\mathcal{F}=\sum_{n=1}^{N}{\bf y}_\mathcal{F}^*(n){\bf x}_\mathcal{F}(n)({\lambda_n^*})^{i-1},\label{eq-Rxy}
\end{align}
which has lower computational complexity than calculating the autocorrelation and the cross-correlation using the definition directly. Note that ${\bf y}_\mathcal{F}={\bf V}^{-1}{\bf y}$ and  ${\bf x}_\mathcal{F}={\bf V}^{-1}{\bf x}$ are the Fourier (DFT) representations of the output and input signals, respectively, and $({\bf \Lambda}^H)^k{\bf \Lambda}^k={\bf I}$.

The optimal LSI filtering \eqref{SimplifiedSolution} also has power spectrum representation, if $L=N$. Note that we can rewrite \eqref{eq-Ryy} in matrix form as
\begin{align}\label{eq-Ryy-W}
\bR_{\by\by} = \bW_\lambda \bY_\mathcal F \bW_\lambda^H,
\end{align}
where $\bW_\lambda$ is a Vandermonde matrix with $W_\lambda(i,j) = (\lambda^*_{j})^{i-1}$, i.e.,
\begin{align}\label{eq-W_lambda}
\bW_\lambda = \begin{pmatrix}
1 			& 1 			& \cdots 	& 	1 \\
\lambda^*_1 	& \lambda^*_2	& \cdots	& 	\lambda^*_{N} \\
\vdots		& \vdots	& \vdots	& \vdots	\\
(\lambda^*_1)^{N-1} 	& (\lambda^*_2)^{N-1}	& \cdots	& 	(\lambda^*_{N})^{N-1} \\
\end{pmatrix}
\end{align}
and
\begin{align}
\bY_\mathcal F = \text{diag}(|{\bf y}_\mathcal{F}(n)|^2, n=1,\cdots,N).
\end{align}
Given $\lambda_i = e^{-j\frac{2\pi(i-1)}{N}}$, it is easy to see that $\bW_\lambda^H = \sqrt{N} \bW$, where $\bW$ is the DFT matrix. 
Then \eqref{eq-WienerHopf1} becomes
\begin{align}
N \bW^H \bY_\mathcal F \bW \bh= {\bf r}{\bf _{xy}},
\end{align}
i.e.,
\begin{align}
\bW \bh = \frac{1}{N}\bY_{\mathcal F}^{-1} \bW {\bf r}{\bf _{xy}}.
\end{align}
Thus
\begin{align}
\bh_{\mathcal F} = \frac{1}{N}\bY_{\mathcal F}^{-1} {\bf r}{\bf _{xy,{\mathcal F}}}.
\end{align}
That is
\begin{align}\label{eq-PowerSpectral}
{\bf h}_\mathcal{F}(i)=\frac{{\bf r}{\bf _{xy}}_{\mathcal{F}}(i)}{N|{\bf y}_\mathcal{F}(i)|^2}.
\end{align}
Here the subscript $\mathcal{F}$ represents the DFT.
This result is consistent with the power spectrum interpretation in the classical signal processing. Note that the property of $\boldsymbol \Lambda_e$ 
is a key for the spectrum representation \eqref{eq-PowerSpectral} to hold.

We will show that a similar structure exists for a general LSI filter for a graph with special structures.

\subsection{ Correlation Functions and Optimal (Wiener) Filtering for Arbitrary Graph Signals} \label{sec-graphwinerfilter}

Arbitrary graph signals may have complex structures, e.g., directed or undirected, weighted or un-weighted, etc. As we defined the shift matrix to be the ${{\bf A}_{\phi}}$, we can construct a general LSI filter as a polynomial of the shift matrix, i.e., ${\bf H}=h({{\bf A}_{\phi}})=\sum_{k=0}^{L-1}h_k{\bf A}_{\phi}^k$, where $h_k$ is the $k$-th filter tap. We also define ${\bf h}\triangleq[h_0 \;\; h_1 \cdots h_{L-1}]^T$ as the vector of the filter. Consider again the denoising problem \eqref{Opt11} given by 
\begin{align}\label{ModifiedDenoising}
\min_{\bf h} \;\|{\bf B}_{\rm new}{\bf h}-\bf x\|_2^2,
\end{align}
where ${\bf B}_{\rm new}=[{\bf y} \;\; {{\bf A}_{\phi}}{\bf y} \cdots {{\bf A}_{\phi}}^{L-1}{\bf y}]$ and the solution to such a problem was obtained earlier. We proved that for the time series graph, the optimal solution is the Wiener filter given by \eqref{SimplifiedSolutionGraph}.
\begin{figure*}
\begin{align}\label{SimplifiedSolutionGraph}
&\begin{pmatrix}
{\bf y}^H{\bf y}&& {\bf y}^H{\bf A}_{\phi}{\bf y}&&\!\!\!\!\!\!\!\!\! \cdots\!\!\!\!\!\!\! &&{\bf y}^H({\bf A}_{\phi})^{L-1}{\bf y} \\{\bf y}^H{\bf A}_{\phi}^H{\bf y} && {\bf y}^H{\bf A}_{\phi}^H{\bf A}_{\phi}{\bf y} && \!\!\!\!\!\!\!\!\! \cdots\!\!\!\!\!\!\!&&{\bf y}^H{\bf A}_{\phi}^H({\bf A}_{\phi})^{L-1}{\bf y}  \\ \vdots&& \vdots && \!\!\!\!\!\!\!\!\! \ddots\!\!\!\!\!\!\!&&\vdots \\{\bf y}^H({\bf A}_{\phi}^{L-1})^H{\bf y} &&{\bf y}^H({\bf A}_{\phi}^{L-1})^H{\bf A}_{\phi}{\bf y}&&\!\!\!\!\!\!\! \cdots\!\!\!\!\!&&{\bf y}^H({\bf A}_{\phi}^{L-1})^H({\bf A}_{\phi}^{L-1}){\bf y}
\end{pmatrix}\!\!{\bf h}\!=\!\!\!\begin{pmatrix}
{\bf y}^H{\bf x} \\{\bf y}^H{\bf A}_{\phi}^H{\bf x} \\ \vdots \\{\bf y}^H({\bf A}_{\phi}^{L-1})^H{\bf x}
\end{pmatrix}.
\end{align}
\end{figure*}
The Wiener filter structure depends on the availability of the autocorrelation function of the output data $\bf y$ and the cross-correlation of the input $\bf x$ and output $\bf y$. For a general graph, the autocorrelation and cross-correlation need to be defined on a particular graph structure.

$\bf Definition$ : We define the autocorrelation function of the signal ${\bf y}$ on an arbitrary graph with lags $l$ and $m$, i.e., the correlation between the shifted version with lag $l$, $\bA_{\phi}^{l}{\by}$, and the shifted version with lag $m$, $\bA_{\phi}^{m}{\by}$, as  
\begin{align}\label{eq-gautocorr}
R_{\bf yy}^{G}(l,m) &\triangleq  {\bf y}^H({\bf A}_{\phi}^{l})^{H} {\bf A}_{\phi}^{m}{\bf y} \nonumber \\
&=\by^H(\bV\bLambda_\phi^l \bV^{-1})^H (\bV\bLambda_\phi^m \bV^{-1}) \by \nonumber \\
&=(\bV^{-1}\by)^H (\bLambda_\phi^*)^l \bV^H \bV \bLambda_\phi^m \bV^{-1} \by. 
\end{align}
We also define  the cross-correlation between the input vectors $\bf x$ and the output vector $\bf y$ at lag $l$, as 
\begin{align}\label{eq-gcrosscorr}
r_{\bf xy}^{G}(l)&\triangleq {\bf y}^H({\bf A}_{\phi}^l)^{H} {\bf x} \nonumber \\
&= \by^H(\bV\bLambda_\phi^l \bV^{-1})^H \bx \nonumber \\
&= (\bV^{-1}\by)^H (\bLambda_\phi^*)^l \bV^H \bx.
\end{align}
As can be seen, the autocorrelation and the cross-correlation on graphs are indeed the correlations between a graph signal and a graph shifted signal, where a shift is defined by the graph shift operator.

The linear equations \eqref{SimplifiedSolutionGraph} can hence be rewritten as
\begin{equation}\label{GraphWienerEQ}
{\bf R}^G_{{\bf yy}}{\bf h}={\bf r}^G_{\bf xy},
\end{equation}
i.e.,
{\small
\begin{align}\label{eq-WienerHopfMatGraph}
&\begin{pmatrix}
R_{\bf yy}^{G}(0,0)		& {R_{\bf yy}^{G}}(0,1)	& \cdots 	&{R_{\bf yy}^{G}}(0,L-1) \\
R_{\bf yy}^{G}(1,0) 	& R_{\bf yy}^{G}(1,1) 	& \cdots & {R_{\bf yy}^{G}}(1,L-2) \\
 \vdots& \vdots & \ddots & \vdots \\
{R_{\bf yy}^{G}}(L-1,0) &{R_{\bf yy}^{G}}(L-1,1)&\cdots& R_{\bf yy}^{G}(L-1,L-1) \\
\end{pmatrix}{\bf h} \nonumber \\
&=\begin{pmatrix}
r_{\bf xy}^{G}(0) \\
r_{\bf xy}^{G}(1) \\
\vdots \\
r_{\bf xy}^{G}(L-1)
\end{pmatrix},
\end{align}
}
which can be considered the Wiener-Hopf equation for graph signals. Note that the autocorrelation matrix ${\bf R}^G_{{\bf yy}}$ is generally not a Toeplitz matrix for a directed graph.

We note that, a typical graph can have a large number of nodes, meaning that the size of the ${\bf A}_{\phi}$ matrix is large. The optimal filtering \eqref{eq-WienerHopfMatGraph} needs the autocorrelation and cross-correlation functions. Since the computational complexity of calculating these functions are high (see the definition of the functions and the large multiplications of matrices), it is desirable to obtain the autocorrelation and cross-correlation functions in \eqref{eq-WienerHopfMatGraph}, using similar equations as \eqref{eq-Ryy} and \eqref{eq-Rxy}. To do so, let us consider ${{\bf A}_{\phi}}={\bf V}{\bf \Lambda}_\phi{\bf V}^{-1}$, where
\begin{align}
{\bf \Lambda}_\phi=\begin{pmatrix}
\lambda_1& 0 & 0 &\cdots & 0\\
0 & \lambda_2 & 0& \cdots& 0\\
\vdots&\vdots & \ddots& \cdots & \vdots\\
0 &0 & \cdots & \lambda_{N-1}& 0
\\
0&0&0&\cdots &\lambda_{N}
\end{pmatrix}.
\end{align}

Now if $\bV$ is unitary, i.e., $\bV^H = \bV^{-1}$, true for all undirected graphs with real-valued $\bA$ and the cyclic graph, then $\bA_{\phi}$ is also unitary and we have
\begin{align} 
 {\bf R}_{\bf yy}^G(i,j)&=R_{\bf yy}^{G}(i-1,j-1) \nonumber \\
 &=\sum_{n=1}^{N}{\bf y}_{\mathcal F}^*(n) {\bf y}_{\mathcal F}(n) (\lambda_n^*)^{i-1}\lambda_n^{j-1} \nonumber \\
 &=\sum_{n=1}^{N}{\bf y}_{\mathcal F}^*(n) {\bf y}_{\mathcal F}(n) (\lambda_n^*)^{i-1}\lambda_n^{j-1} \nonumber \\
 &=\sum_{n=1}^{N}|{\bf y}_{\mathcal F}(n)|^2 \lambda_n^{j-i}, \label{eq-RGyy} \\
{ {\bf r}^G{\bf _{xy}}}(i)&=r_{\bf xy}^{G}(i-1)=\sum_{n=1}^{N}{\bf y}_{\mathcal F}^*(n){\bf x}_{\mathcal F}(n){(\lambda_{n}^*)}^{i-1}.\label{eq-rGxy}
\end{align} 
Also note that, since $\lambda_n^*\lambda_n=1$ for $\bA_{\phi}$, 
\begin{equation}
(\lambda_n^*)^i\lambda_n^{j}=\lambda_n^{j-i} = (\lambda_n^*)^{i-j}.
\end{equation}

Now the autocorrelation matrix ${\bf R}^G_{{\bf yy}}$ is now a Toeplitz matrix and the solution to \eqref{GraphWienerEQ} becomes similar to the Wiener filter for time series data. And we can get similar equations as \eqref{eq-Ryy-W}:
\begin{align}\label{eq-RGyy-W}
\bR^G_{\by\by} = \bW_\lambda \bY_\mathcal F \bW_\lambda^H,
\end{align}
where $\bW_\lambda$ is a Vandermonde matrix with $W_\lambda(i,j) = (\lambda^*_{j})^{i-1}$, defined in \eqref{eq-W_lambda}, and $\bY_\mathcal F = \text{diag}(|{\bf y}_\mathcal{F}(n)|^2, n=1,\cdots,N)$.

Now if we use the new shift operator $\bA_e$, where $\lambda_i = e^{-j\frac{2\pi(i-1)}{N}}$, we also have $\bW_\lambda^H = \sqrt{N} \bW$, where $\bW$ is the DFT matrix . 
Then \eqref{GraphWienerEQ} becomes
\begin{align}
N \bW^H \bY_\mathcal F \bW \bh= {\bf r}{\bf _{xy}}.
\end{align}
Thus
\begin{align}
\bh_{DFT} = \frac{1}{N}\bY_{\mathcal F}^{-1} {\bf r}_{\bx\by,DFT}.
\end{align}
That is
\begin{align}\label{eq-PowerSpectralGraph}
\bh_{DFT}(i)=\frac{{\bf r}_{\bx\by,DFT}(i)}{N|{\bf y}_\mathcal{F}(i)|^2},
\end{align}
where $\bh_{DFT}=\bW \bh$, ${\bf_{xy}}_{DFT}= \bW \bf_{xy}$, and $\by_\mathcal{F} = \bV^{-1} \by$. Note that the subscript $DFT$ represents ordinary DFT and the subscript $\mathcal{F}$ represents the GFT.

{\emph{Remark:}}  The simple closed-form power spectrum solution \eqref{eq-PowerSpectralGraph} of optimal graph Wiener filter only holds under two conditions: (i) The matrix $\bV$ consisting of eigenvectors of graph adjacency matrix is unitary. All undirected graphs satisfy this condition. Some directed graphs such as the cyclic graph also satisfy this condition. Note that the autocorrelation matrix $\bR_{\by\by}$ is a Toeplitz matrix with this condition. (ii) The new shift operator $\bA_e$ is applied. Any other graph shift operator does not lead to such efficient solution for the filtering problem such as \eqref{ModifiedDenoising}. 
This is a major difference between our shift operator \eqref{eq-lambda_phi} and the shift operator proposed by \cite{GraphTranslation,StationaryGraphSignal,GiraultThesis}.

For general directed graphs, the optimal filtering solution can only be obtained by solving the general graph Wiener-Hopf equation \eqref{eq-WienerHopfMatGraph}. \label{label-remark}

Note that a graph can represent an intrinsic network correlation structure that is a sophisticated generalization of a linear time shift correlation structure in classical signal processing. However, it is non-trivial to transfer all the classical signal processing results on a linear time shift to a complex graph correlation structure. The new definition of auto/cross-correlation and graph Wiener filtering are our efforts to develop graph signal processing tools in parallel to classical signal processing, and are also a major contribution of this paper. \label{WienerRemark}

\section{Simulations}\label{sec-simulation}
Fig.~\ref{fig:RandomSensorNet} shows a random deployment of 20 sensor nodes in a two-dimensional area. The coordinates of each node is randomly generated in a square area, where $x$ and $y$ limits are in $[0,1]$. To construct an undirected graph, the $6$ nearest neighbor rule is used -- each node is connected to $6$ of its nearest neighbors.
The graph adjacency matrix of this random sensor network can be obtained from the connection between the nodes. We assign the weight 1 for each connection for simplicity. We plot the eigenvalues of the graph adjacency matrix in Fig.~\ref{fig:EigenvaluesKSparseUndirected}. The eigenvalues of the graph adjacency matrix are real-valued since the graph is undirected.

\begin{figure} [H]
	\centering
	\centerline{\resizebox{!}{6.4cm}{\includegraphics{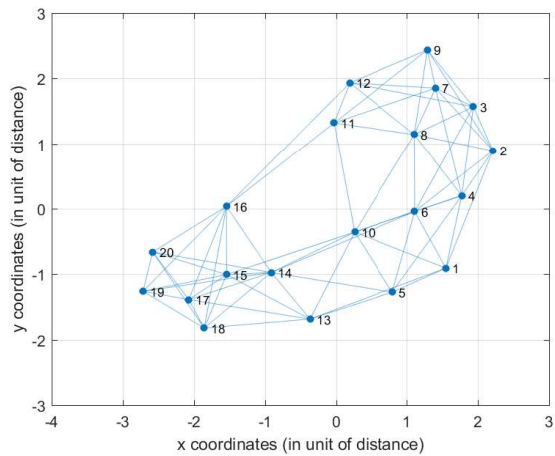}}}
	\caption{Random sensor network with 20 nodes (undirected graph).}
	\label{fig:RandomSensorNet}
\end{figure}

\begin{figure} [H]
	\centering
		\centerline{\resizebox{!}{6.0cm}{\includegraphics{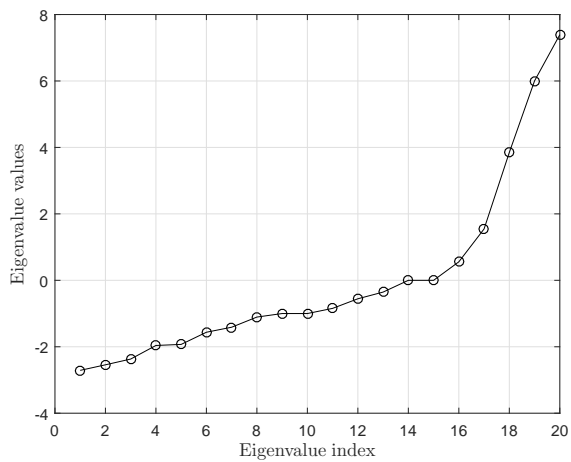}}}
		\caption{Eigenvalues of the graph adjacency matrix (undirected graph).}
		\label{fig:EigenvaluesKSparseUndirected}
\end{figure}

We now define a $K$-sparse graph signal, ${\bf x}$, whose Fourier representation has $K$-nonzero values. That means $\bf x$ has $K$-nonzero component in the Fourier domain, i.e., ${\bf x}_{\cal F}={\bf V}^{-1}{\bf x}$ has $K$-nonzero components. To generate a $K$-sparse graph signal, we randomly generate an $N\times 1$ vector whose first $K$ components are generated from a Gaussian (this can be generated from any distributional function) and the rest of its elements are zero. For instance, we generate a graph signal in the Fourier domain as ${\bf x}=[-0.296 \;-1.497 \;-0.905\;-0.404 \;-0.726\;-0.866\;-0.423\;-0.943\;1.3419\;-0.989\;0\;0\;0\;0\;0\;0\;0\;0\;0\;0]^T$. 
Note that the energies of the graph signal in both shift and Fourier domains are equal to 8.32 ($||{\bf x}||_2^2=||{\bf x}_{\cal F}||_2^2$) for this undirected graph. 

We plot the energy of the shifted versions of this graph signal in Fig.~\ref{fig:ShiftEffect} for two alternatives of graph shift operators; graph adjacency matrix ${\bf A}$ and normalized graph shift matrix (defined in \cite{Moura2}), ${\bf A}_{\rm norm}=\frac{\bf A}{\lambda_{\max}({\bf A})}$, where $\lambda_{\max}({\bf A})$ is the eigenvalue with the largest magnitude. Note that $\bA_{\rm norm}^N \bx = \bV\boldsymbol \Lambda_{\rm norm}^N \bV^{-1}\bx$, where $N$ is the number of shifts. When the graph adjacency matrix $\bA$ is used as shift operator, note that some of its eigenvalues are larger than one. After several shifts, the energy of frequency components corresponding to those eigenvalues larger than one become larger and larger, and those components will dominate the energy of the shifted graph signal. Therefore, the signal energy in the frequency domain increases as the number of shift increases. On the other hand, when the normalized graph shift matrix is used as the shift operator, none of the eigenvalues are larger than one. The energy of frequency components corresponding to those eigenvalues smaller than one are diminishing with increasing number of shifts. The energy of frequency components with eigenvalue equal to one will remain. Therefore, overall the signal energy in the frequency domain decreases as the number of shift increases. 
Note that our defined shift operator ${\bf A}_\phi$ does not change the energy of the graph signal both in shift and Fourier domain for any amounts of shifts.
 \label{label-eigenvalue}

\begin{figure}[b!]
\centering
		\centerline{\resizebox{!}{6.0cm}{\includegraphics{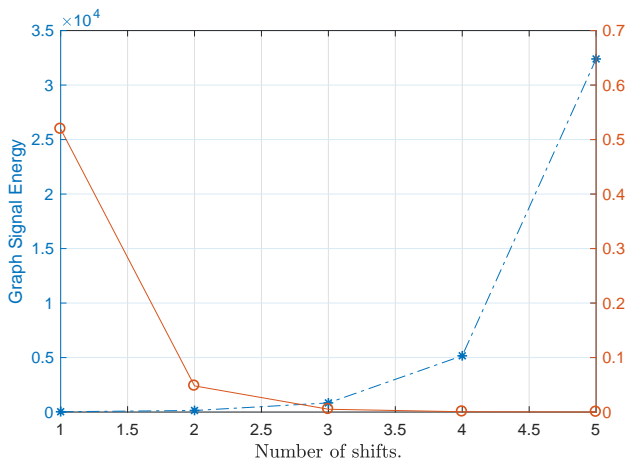}}}
		\caption{Energy of a 10-sparse random signal in the Fourier domain for different shift amounts (the graph structure is undirected). Note that the signal energy using the new shift operator is constant in this case.}
		\label{fig:ShiftEffect}
\end{figure}

Fig.~\ref{fig:DirectedGraph} illustrates a random sensor network with asymmetric graph adjacency matrix, i.e., a directed graph. 
To construct this graph, we randomly sample the graph adjacency matrix of the undirected graph presented in the previous example. We further plotted the magnitude of the eigenvalues of the graph adjacency matrix in Fig.~\ref{fig:EigenvaluesKSparseDirected} (since eigenvalues are generally complex numbers for directed graphs adjacency matrix).

Fig.~\ref{fig:EnergyAmount2} shows the energy of a 10-sparse random signal for different amounts of shift when the graph structure is directed. The signal energy in the Fourier domain is a fixed number, i.e., 8.32. Note that we use our defined graph shift operator in this example to illustrate the difference between the energy of the shifted signal in shift and Fourier domains. To be consistent, we use the same graph signal in the undirected graph example presented earlier in this section. Fig.~\ref{fig:EnergyAmount2} demonstrates that the energy of the graph signal in the Fourier domain is a fix number while in the shift domain is oscillating. Since Fourier operator for a general directed graph is not unitary, the energy of the graph signal in the shift and Fourier domains are different. We can show the oscillating behavior of energy of the shifted graph signal ${\bf x}_k$ in the shift domain as
\begin{align} \label{ProofEnergy}
\|{\bf x}_k\|_2^2&=\|{\bf A}_{e}^k{\bf x}\|_2^2=\|{\bf V}{\boldsymbol \Lambda}_{e}^k{\bf x}_{\cal F}\|_2^2
= \left\|{ \sum_{l=1}^{N}}{\lambda}_{e_l}^k{\bf x}_{\cal F}(l){\bf v}_l \right\|_2^2 \nonumber \\
&= \left\|{ \sum_{l=1}^{N}}{\lambda}_{e_l}^k\hat{\bf v}_l\right\|_2^2 
={ \sum_{l=1}^{N}}{ \sum_{m=1}^{N}}(e^{-j\frac{2 \pi (m-l)k}{N}})\hat{\bf v}_l^H\hat{\bf v}_m,
\end{align}
where $\hat{\bf v}_l = {\bf x}_{\cal F}(l){\bf v}_l$. Note that, the energy of the shifted graph signal in the shift domain is a linear combination of constant terms $\hat{\bf v}_l^H\hat{\bf v}_m$ where the weights are complex-valued and functions of $k$ (the number of shifts). Changing $k$ leads to the variability of the energy in the shift domain. We can apply the frame theory results on frame bounds \cite{FrameTheory} to obtain a lower and an upper bound for the energy in the shift domain, as shown in \eqref{eq-frame1}.
We further note that our shift operator ${\bf A}_e$, although it does not preserve the energy in the shift domain for all amounts of shifts, it preserves the energy in the shift domain for integer multiples of $N$ since ${\bf A}_{e}^{lN}=({\bf A}_{e}^{N})^l=({\bf I})^l={\bf I}$.

\begin{figure}[htbp]
	\centering
	\centerline{\resizebox{!}{6.6cm}{\includegraphics{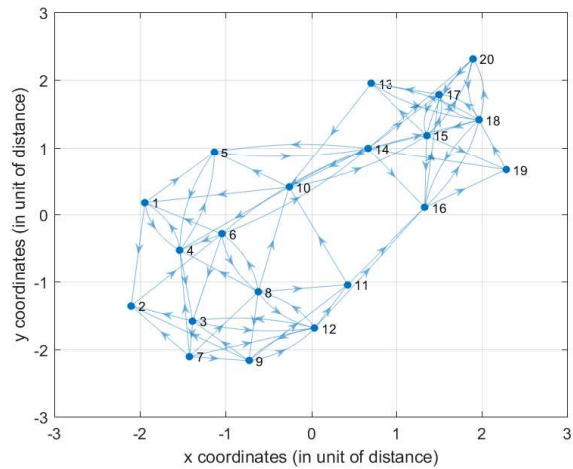}}}
	\caption{Directed sensor network with 20 nodes.}
	\label{fig:DirectedGraph}
\end{figure}

\begin{figure}[htbp]
	\centering
	\centerline{\resizebox{!}{6.2cm}{\includegraphics{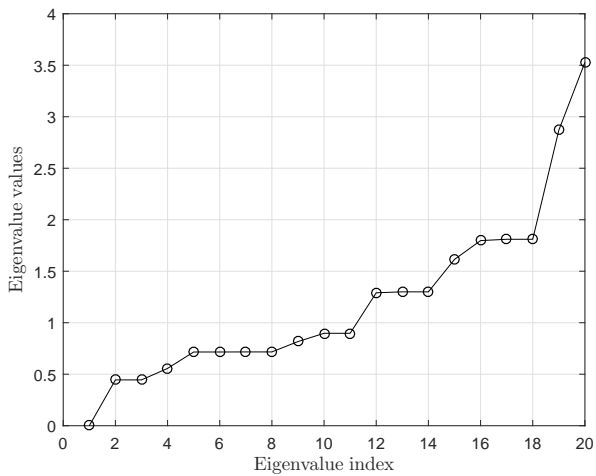}}}
	\caption{Magnitude of the eigenvalues of the graph adjacency matrix (directed graph).}
	\label{fig:EigenvaluesKSparseDirected}
\end{figure}

\begin{figure}[htbp]
	\centering
	\centerline{\resizebox{!}{6.2cm}{\includegraphics{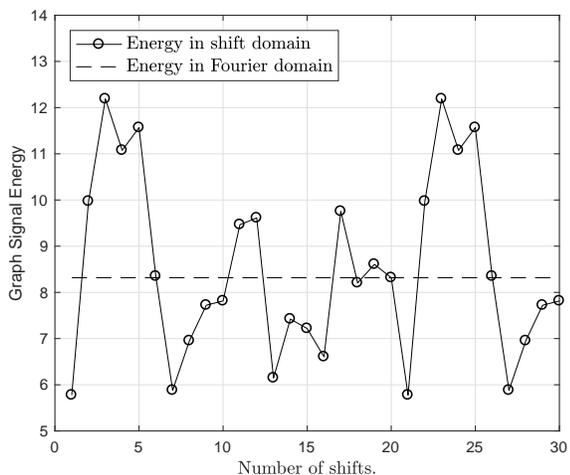}}}
	\caption{Energy of a 10-sparse random signal in shift and Fourier domain using our defined graph shift operator (directed graph).}
	\label{fig:EnergyAmount2}
\end{figure}

Figs.~\ref{fig:ComparisonVar1} and~\ref{fig:ComparisonVar100} show the percentage of the reconstruction error for the Wiener filtering problem \eqref{ModifiedDenoising}, for two different sets of noisy measurements namely; random i.i.d. Gaussian noises with variances $\sigma_n^2 = 1$ and $\sigma_n^2 = 100$. The dataset considered in this example contains the average temperatures of 40 US states capitals. The graph signal ${\bf x}_t$ is a $40 \times 1$ vector and $t \in \{1,2,\cdots,M\}$ for a horizon of $M=264$ consecutive days in 2015, and we consider the noisy measurements of those graph signals as ${\bf y}_t={\bf x}_t+{\bf n}_t$. For the two scenarios considered here, we define the ratio of the average  power of signals of all nodes and the noise variance as \emph{signal-to-noise-ratio} (SNR). The SNR of the first scenario (where the noise variance $\sigma_n^2 = 1$) and the second scenario (where the noise variance $\sigma_n^2 = 100$) are $35.66\; {\rm dB}$ and $15.66\; {\rm dB}$, respectively.

The average percentage of the reconstruction error is defined as $\frac{1}{M}\sum_{t=1}^M\frac{||{\bf x}_t-{\bf H}_t^{\rm o}{\bf y}_t||}{||{\bf x}_t||}$, where ${\bf H}_t^{\rm o}$ is the optimal graph filter obtained by the optimization problem \eqref{ModifiedDenoising}. We consider three different approaches to construct the graph, i.e., the $k$-nearest neighbor method with $k=9$ \cite{Moura2}, the exponentially distance-based weighted graph adjacency considered by \cite{Moura2}, and the empirical covariance-based graph construction introduced in \cite{ChaZhang2}. In all cases, the reconstruction error of Wiener filtering using our shift operator is much lower than the traditional adjacency-based operators. 

The source of errors comes from the following facts. For large values of $L$ (the number of filter coefficients), the polynomial representation of the LSI filter ${\bf H}=\sum_{l=0}^L h_l{\bf A}^l$ dominates only the frequency components that have the largest eigenvalues of ${\bf A}$. More specifically, adding further coefficients to the filter, does not improve the performance of the Wiener filtering since the frequency components corresponding to small eigenvalues are indistinguishable from the noise. This effect can be explained by Fig.~\ref{fig:EigenvaluesOfGraphAdjacencyWeatherData} where we plot the eigenvalues of the graph adjacency matrix of the weather graph. We can observe that the magnitude of eigenvalues of the graph adjacency matrix are not equal to one. Therefore, the performance of the filtering is dominated by the largest eigenvalue(s) of $\bf A$, i.e., adding further coefficients leads to diminishing all eigenvalues (and the corresponding Fourier basis) less than $\lambda_{\max} ({\bf A})$; thereby saturating the performance of the optimal filtering.

We further observe that after almost adding 15 taps (10 taps) when the noise variance is $\sigma_n^2 = 1$, the performance of the adjacency-based approaches is saturated. Note that this number is specific to this example and for different simulation setups this saturating effect may be observed at a different location depending on the magnitudes of the eigenvalues of the graph adjacency matrix $\bf A$. However, if we use our new graph shift operator, adding more coefficients will improve the performance of the filtering operation as it allows us to use more coefficients to reduce the reconstruction error while keeping all frequency components of the graph. Moreover, the cyclic property of our defined graph shift operator, i.e., ${\bf A}_e^N={\bf I}$ ensures that an LSI filter that uses ${\bf A}_e$ as a shift operator needs at most $N$ filter taps to achieve its best performance, while there is no limit on the number of filter taps if other shift operators are used. This signifies the importance of our defined shift operator which offers low complexity (a need for less filter taps to achieve a predefined performance level) and high accuracy. Furthermore, Figs.~\ref{fig:ComparisonVar1} and~\ref{fig:ComparisonVar100} also show that, when the noise variance is larger, the performance of Wiener filtering will be affected and the percentage of the reconstruction error becomes larger, as we expect. One more important note is that for the case where the noise variance  $\sigma_n^2 = 100$, a noise level comparable to the level of the graph signal, we observe that the performance of Wiener filtering for the traditional adjacency-based operators saturates at earlier number of filter taps than the case where the noise variance is  $\sigma_n^2 = 1$. This shows that the performance of Wiener filtering is noise-level dependent for the traditional adjacency-based operators, while the performance of the Wiener filtering using the new shift operator is not noise level dependent.

\begin{figure}[H]
	\centering
	\centerline{\resizebox{!}{6.2cm}{\includegraphics{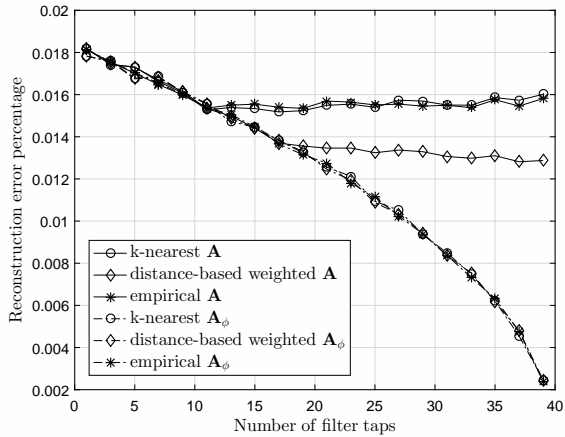}}}
	\caption{Average reconstruction error for the Wiener filtering problem \eqref{ModifiedDenoising}, where the SNR is $35.66 \; {\rm dB}$ (the noise variance $\sigma_n^2 = 1$).}
	\label{fig:ComparisonVar1}
\end{figure}

\vspace{-0.1in}
\begin{figure}[H]
	\centerline{\resizebox{!}{6.2cm}{\includegraphics{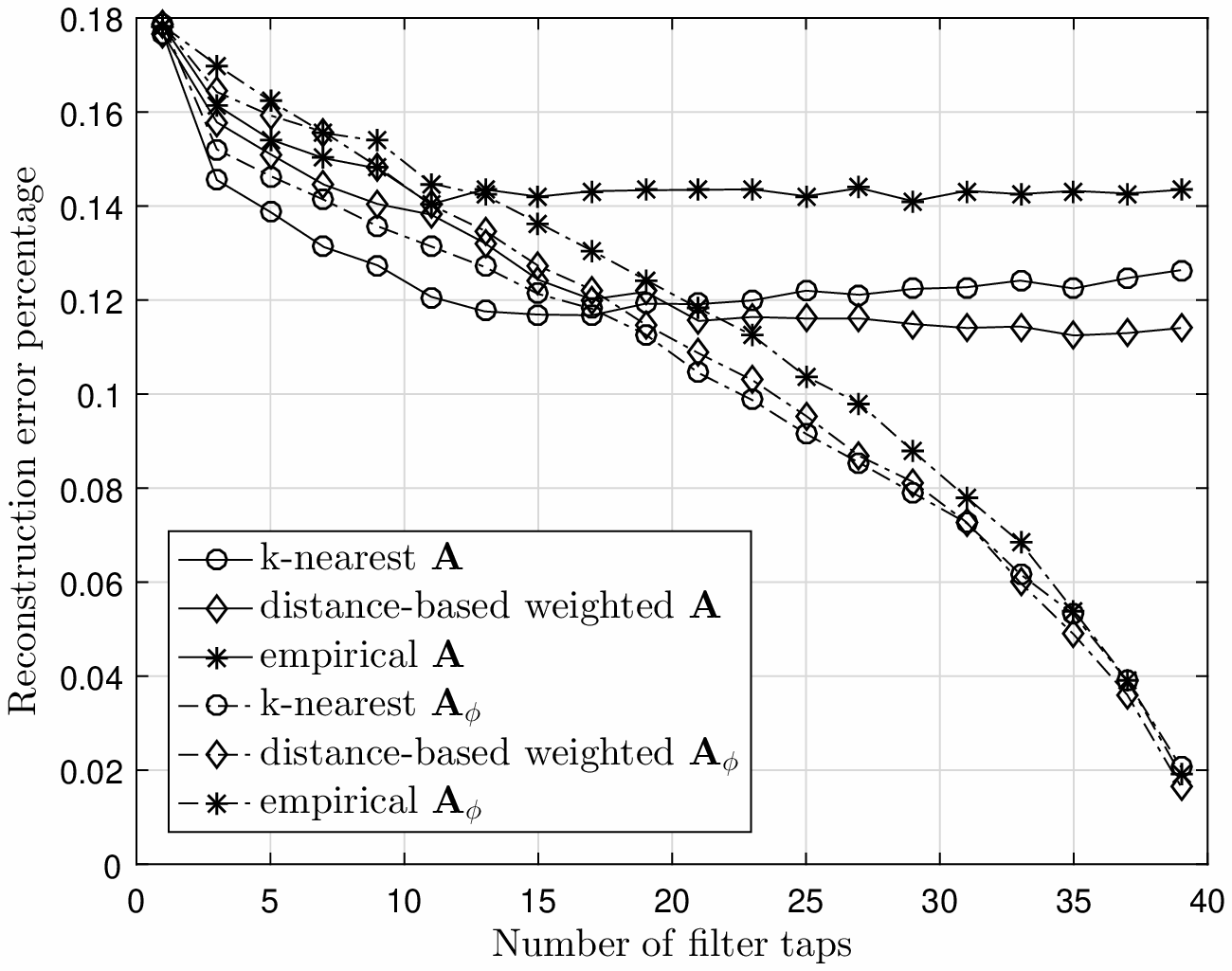}}}
	\caption{Average reconstruction error for the Wiener filtering problem \eqref{ModifiedDenoising}, where the SNR is $15.66 \; {\rm dB}$ (the noise variance $\sigma_n^2 = 100$).}
	\label{fig:ComparisonVar100}
\end{figure}

\vspace{-0.1in}
\begin{figure}[H]
	\centering
	\centerline{\resizebox{!}{6cm}{\includegraphics{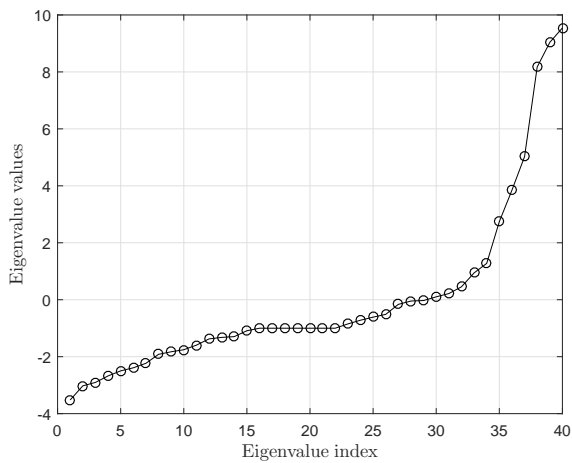}}}
	\caption{Eigenvalues of the graph adjacency matrix of the weather network.}
	\label{fig:EigenvaluesOfGraphAdjacencyWeatherData}
\end{figure}

\section{Conclusions}

In this paper, we define a new set of shift operators for graph signals satisfying the energy-preserving properties as in classical signal processing by resetting the eigenvalues of the adjacency matrix or the Laplacian matrix flexibly. We show that such shift operators preserve the energy content of the graph signal in both shift and frequency domains. We further investigate the properties of LSI graph filters and show that any LSI filter can be written as a polynomial of the graph shift operator. We then categorize the LSI filters as GFIR and GIIR filters, similar to the classical FIR and IIR filters and obtain explicit forms for such filters. Based on these energy-preserving shift operators, we further define autocorrelation and cross-correlation functions of signals on the graph. We then obtain the structure of the optimal LSI graph filters, i.e., Wiener filtering, through the construction of the Wiener-Hopf equation on graph. We show that only with the proposed graph shift operator, we can possibly obtain the efficient spectra analysis and frequency domain filtering in parallel with those in classical signal processing. 
Several illustrative simulations are presented to validate the performance of designed optimal LSI filters.

Our new shift operator based GSP framework enables the signal analysis along a correlation structure defined by a graph shift manifold as opposed to classical signal processing operating on the assumption of the correlation structure with a linear time shift manifold.

\section*{Acknowledgment}
We would like to thank Dr. Jos\'e Moura and Dr. Antonio Ortega for discussions and feedbacks of many concepts and ideas in the early stage of manuscript development.

\setcounter{equation}{0}
\appendix
\numberwithin{equation}{section}

\subsection{Proof of Theorem~\ref{TheoremReversePolynomial}} \label{ReversePolynomial}
Let us assume that ${\bf A}_{\phi}$ can be represented as a polynomial of $\bf A$, i.e.,
\begin{align}
{\bf A}_{\phi}=\sum_{k=0}^{N-1}g_k{{\bf A}}^k.
\end{align}
Since the Fourier basis of $\bf A$ and ${\bf A}_\phi$ are the same by the definition, we can diagonalize the two operator by multiplying $\bf V$ and ${\bf V}^{-1}$ to the two operator form right and left, respectively. Therefore, we can write
\begin{align}
{\boldsymbol \Lambda}_{\phi}=\sum_{k=0}^{N-1}g_k{{\boldsymbol \Lambda}}^k.
\end{align}
This equation can also be written as a linear matrix equation as
\begin{align}
\underset{{\boldsymbol {Z}}}{\underbrace{\begin{pmatrix}
		1& \lambda_1 & \lambda_1^2 &\cdots & \lambda_1^{N-1}\\
		1& \lambda_2 & \lambda_2^2 &\cdots & \lambda_2^{N-1}\\
		\vdots &\vdots & \vdots & \ddots& \vdots\\
		1& \lambda_N & \lambda_N^2 &\cdots & \lambda_N^{N-1}
		\end{pmatrix}}}\begin{pmatrix} g_0 \\g_1 \\ \vdots \\
g_{N-1} \end{pmatrix}=\begin{pmatrix} \lambda_{\phi_1} \\\lambda_{\phi_2} \\ \vdots \\
\lambda_{\phi_N} \end{pmatrix},
\end{align}
or more compactly, as
\begin{align}\label{FilterDesign}
{\bf Z}{\bf g}={\boldsymbol \lambda}_\phi
\end{align}
where ${\bf g}=[g_0 \; g_1 \cdots \; g_{N-1}]^T$ and ${\boldsymbol \lambda}=[\lambda_{\phi_1} \; \lambda_{\phi_2}\; \cdots \; \lambda_{\phi_N}]^T$. We note that, $\bf Z$ is the well-known Vandermonde matrix and has full-rank iff $\lambda_{i} \neq \lambda_{j}$, for $i$, $j \in \{ 1,2,\cdots, N\}$. Therefore, if $\bf Z$ is full-rank, then the linear equation \eqref{FilterDesign} has a unique solution ${\bf g}={\bf Z}^{-1}{\boldsymbol \lambda}$ (We emphasize that, there exists efficient recursive algorithms for obtaining the inverse of Vandermonde matrices with low computational complexity). This completes the proof.

\subsection{Proof of Theorem~\ref{T3}} \label{Theorem3}
The minimal polynomial of an $N \times N$ matrix ${\bf A}_{\phi}$ is defined by a polynomial with minimum degree that satisfies $p({\bf A}_{\phi})=\sum_{i=0}^{L_{{\bf A}_{\phi}} }\alpha_i {\bf A}_{\phi}^i={\bf 0}_{N \times N}$. Note that for the degenerate case where ${\bf A}_\phi$ is not full-rank, ${L_{{\bf A}_{\phi}} } \neq N$. Without loss of generality, we assume that $\alpha_{L_{{\bf A}_{\phi}}}=1$. In order to obtain the other $\alpha_i$'s, we first note that $\sum_{i=0}^{L_{{\bf A}_{\phi}}-1 }\alpha_i {\bf A}_{\phi}^i=-{\bf A}_{\phi}^{L_{{\bf A}_{\phi}} }$. We then write the diagonalized version of this equation as
\begin{align}
\sum_{i=0}^{L_{{\bf A}_{\phi}}-1 }\alpha_i {\boldsymbol \Lambda}_{\phi}^i=-{\boldsymbol \Lambda}_{\phi}^{L_{{\bf A}_{\phi}} }.
\end{align}
One can write this equation in a matrix form as
\begin{align}
\underset{{\boldsymbol {\tilde{Z}}}}{\underbrace{\begin{pmatrix}
		1& \!\!\!\!\lambda_{\phi_1} &\!\!\!\!\!\! \lambda_{\phi_1}^2 &\!\!\!\!\!\!\cdots & \lambda_{\phi_1}^{L_{{\bf A}_{\phi}} -1}\\
		1&\!\!\!\! \lambda_{\phi_2} &\!\!\!\!\!\! \lambda_{\phi_2}^2 &\!\!\!\!\!\!\cdots & \lambda_{\phi_2}^{L_{{\bf A}_{\phi}} -1}\\
		\vdots &\!\!\!\!\vdots &\!\!\!\!\!\! \vdots &\!\!\!\!\!\! \ddots& \vdots\\
		1&\!\!\!\! \lambda_{\phi_{L_{{\bf A}_{\phi}} }} &\!\!\!\!\!\! \lambda_{\phi_{L_{{\bf A}_{\phi}} }}^2 &\!\!\!\!\!\!\cdots & \lambda_{\phi{L_{{\bf A}_{\phi}}}} ^{{L_{{\bf A}_{\phi}}-1}}
		\end{pmatrix}}}\!\!\!\begin{pmatrix} \alpha_0 \\\alpha_1 \\ \vdots \\
\alpha_{{{L_{{\bf A}_{\phi}}-1}}} \end{pmatrix}\!\!=-\!\!\begin{pmatrix} \lambda_{\phi_1}^{L_{{\bf A}_{\phi}}} \\ \lambda_{\phi_2}^{L_{{\bf A}_{\phi}}} \\ \vdots \\
\lambda_{\phi_{{L_{{\bf A}_{\phi}}-1}}}^{L_{{\bf A}_{\phi}}} \end{pmatrix},
\end{align}
or in a more compact form $\tilde{\bZ} {\boldsymbol \alpha}=\tilde{\boldsymbol \lambda}_{\phi}$, where ${\boldsymbol \alpha}=[\alpha_0 \; \alpha_1 \cdots \; \alpha_{N-1}]^T$ and $\tilde{\boldsymbol \lambda}_{\phi}=-[\lambda_{\phi_1}^{L_{{\bf A}_{\phi}}} \; \lambda_{\phi_2}^{L_{{\bf A}_{\phi}}}\; \cdots \; \lambda_{\phi_{L_{{\bf A}_{\phi}}}}^{L_{{\bf A}_{\phi}}}]^T$. The solution to the $\alpha$'s can be easily obtained as $ {\boldsymbol \alpha}=\tilde{\bZ}^{-1}\tilde{\boldsymbol \lambda}_{\phi}$. 

We also note that for the special ${\bf A}_e$, we choose $\lambda_{e_k}=e^{-j \frac{2 \pi (k-1)}{N}}$, $\forall k \in \{1,\cdots, N\}$, if ${\bf A}$ is full-rank, otherwise we choose $\lambda_{e_k}=e^{-j \frac{2 \pi (k-1)}{L_{{\bf A}_{e}}}}$. We therefore obtain the closed-form solution for the ${\boldsymbol \alpha}$ as
$\alpha_k=-\sum_{l=1}^N e^{j \frac{2 \pi (k-1)}{L_{{\bf A}_{e}}}} \lambda_{e_l}^{L_{{\bf A}_{e}}}$. That is $\alpha_0 = -1$ and $\alpha_l=0$ for $l \ne 0$. This is also straightforward since $\lambda_{e_k}^{L_{{\bf A}_{e}}} =1, \forall k$, i.e., i.e., $\bA_e^{L_{\bA_e}}= \bI$.

Once we obtain $\alpha_k$, we can write $ {\bf A}_{\phi}^{L_{{\bf A}_{\phi}}}=-\sum_{i=0}^{L_{{\bf A}_{\phi}}-1 }\alpha_i {\bf A}_{\phi}^i$. 
 
Now consider an LSI filter with length $L=L_{{\bf A}_{\phi}}+1$ as ${\bf H} =\sum_{k=0}^{L}h_k{{\bf A}_{\phi}}^k$. Using the results obtained earlier, we can rewrite this LSI filter as
\begin{align}
{\bf H} &=\sum_{k=0}^{L_{{\bf A}_{\phi}}}h_k{{\bf A}_{\phi}}^k 
=\sum_{k=0}^{L_{{\bf A}_{\phi}}-1}h_k{{\bf A}_{\phi}}^k+h_{L_{{\bf A}_{\phi}}}{{\bf A}_{\phi}}^{L_{{\bf A}_{\phi}}}\nonumber \\ 
&=\sum_{k=0}^{L_{{\bf A}_{\phi}}-1}\!\!h_k{{\bf A}_{\phi}}^k+h_{L_{{\bf A}_{\phi}}}(-\!\!\!\sum_{k=0}^{L_{{\bf A}_{\phi}}-1 }\!\!\alpha_k {\bf A}_{\phi}^k)\nonumber \\ &=\sum_{k=0}^{L_{{\bf A}_{\phi}}-1}(h_k-h_{L_{{\bf A}_{\phi}}}\alpha_k ){{\bf A}_{\phi}}^k 
=\sum_{k=0}^{L_{{\bf A}_{\phi}}-1}\breve{h}_k{\bf A}_{\phi}^k.
\end{align}
This procedure can be repeated for any arbitrary $L> L_{{\bf A}_{\phi}}+1$. Therefore, we only consider LSI graph filters with $L_{{\bf A}_{\phi}}$ filter taps and we interchangeably use $L$ instead of $L_{{\bf A}_{\phi}}$.  

We finally note that, for the special case ${\bf A}_e^{L_{{\bf A}_{\phi}}}={\bf I}$, meaning that if $k>{L_{{\bf A}_{\phi}}}$, then ${\bf A}_e^{k}={\bf A}_e^{m}$ where $m=(k \mod {L_{{\bf A}_{\phi}}})$.

\bibliographystyle{IEEEtran}
\bibliography{IEEEabrv,reference}

\begin{IEEEbiography}[{\includegraphics[width=1in,height=1.25in,clip,keepaspectratio]{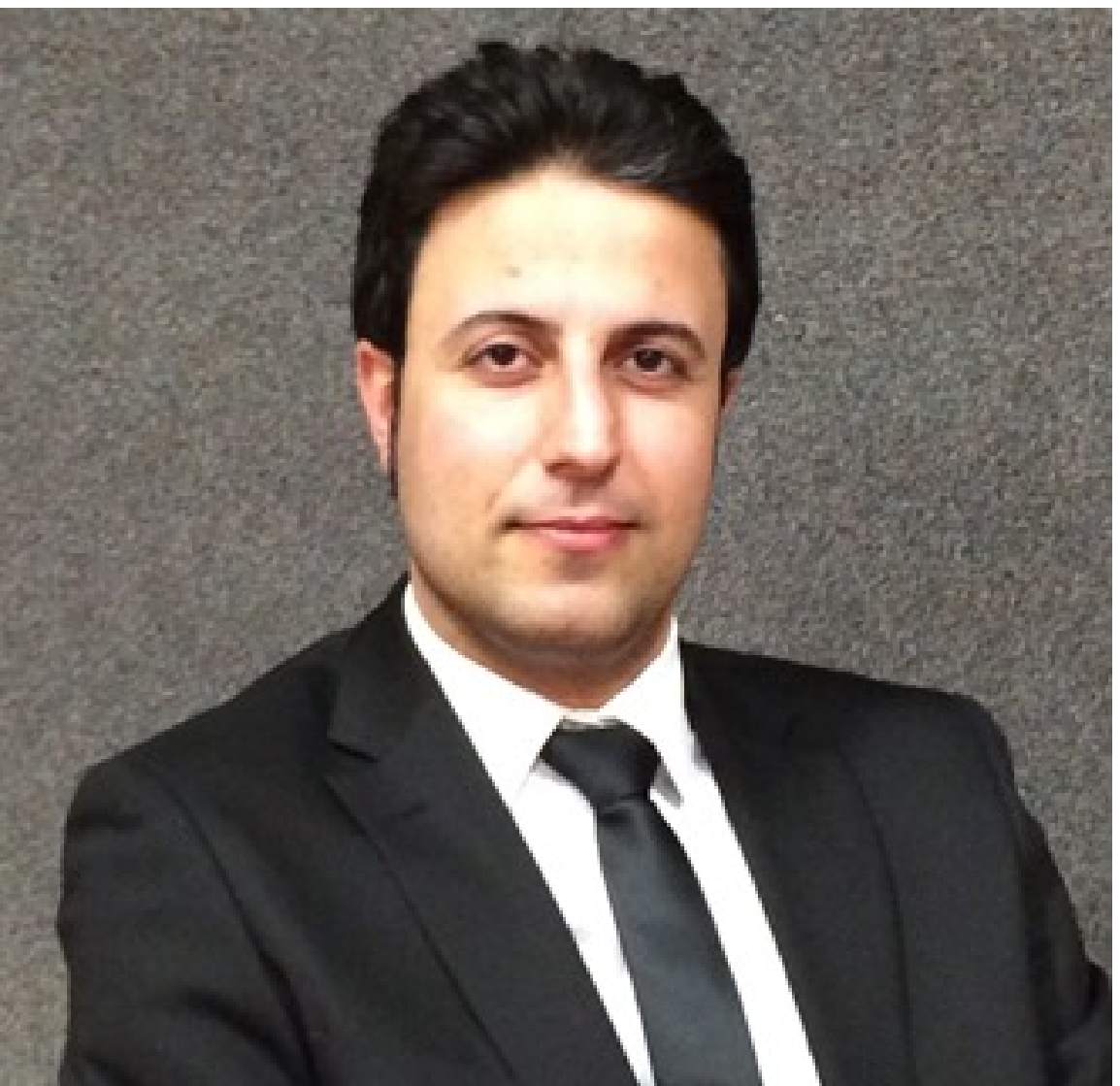}}] {Adnan Gavili} was born in Sanandaj, Kurdistan, Iran. He received the B.Sc., M.Sc. and PhD degrees in electrical engineering from Iran University of Science and Technology (IUST), Sharif University of Technology (SUT) and University of Ontario Institute of Technology (UOIT) in 2009, 2011 and 2015, respectively. From September 2015 to August 2016, he has worked on Graph Signal Processing (GSP) with Dr. Xiao-Ping Zhang as a post-doctoral research fellow at Ryerson University. Since 2013, he has served as a reviewer for \textit{IEEE Transactions on Signal Processing, IEEE Transactions on Wireless Communications} and \textit{IEEE Signal Processing Letters}. His research interests include GSP, classical signal processing, cooperative green communications, cognitive radio networks and signal processing applications in finance.
\end{IEEEbiography}

\begin{IEEEbiography}[{\includegraphics[width=1in,height=1.25in,clip,keepaspectratio]{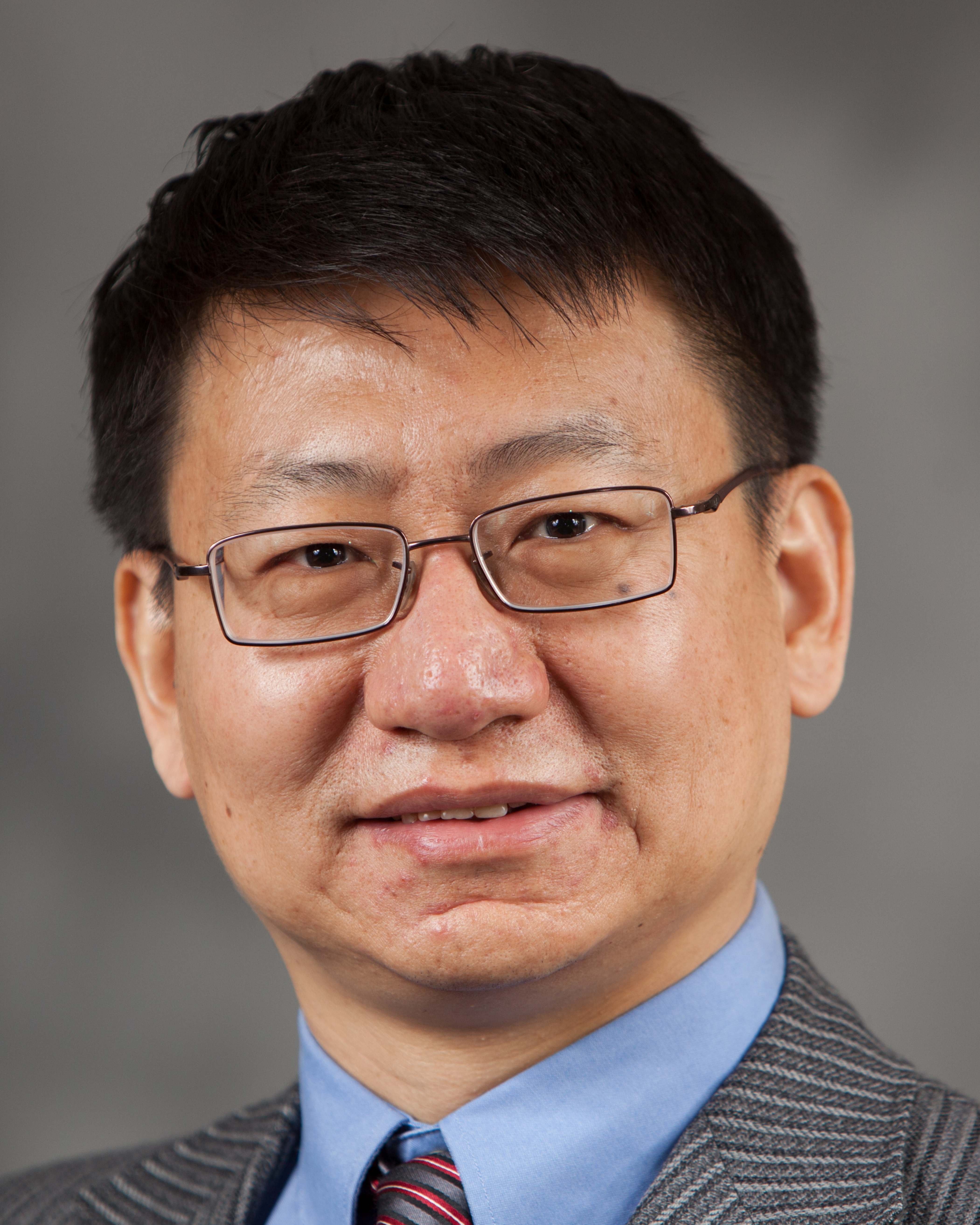}}]
{Xiao-Ping Zhang (M'97, SM'02)} received B.S. and Ph.D. degrees from Tsinghua University, in 1992 and 1996, respectively, both in Electronic Engineering. He holds an MBA in Finance, Economics and Entrepreneurship with Honors from the University of Chicago Booth School of Business, Chicago, IL. 

Since Fall 2000, he has been with the Department of Electrical and Computer Engineering, Ryerson University, where he is now Professor, Director of Communication and Signal Processing Applications Laboratory (CASPAL). He has served as Program Director of Graduate Studies. He is cross appointed to the Finance Department at the Ted Rogers School of Management at Ryerson University. His research interests include statistical signal processing, image and multimedia content analysis, sensor networks and electronic systems, machine learning, and applications in bioinformatics, finance, and marketing. He is a frequent consultant for biotech companies and investment firms. He is cofounder and CEO for EidoSearch, an Ontario based company offering a content-based search and analysis engine for financial data. 

Dr. Zhang is a registered Professional Engineer in Ontario, Canada, and a member of Beta Gamma Sigma Honor Society. He is the general co-chair for ICASSP2021. He is the general co-chair for 2017 GlobalSIP Symposium on Signal and Information Processing for Finance and Business. He is an elected member of ICME steering committee. He is the general chair for MMSP'15. He is the publicity chair for ICME'06 and program chair for ICIC'05 and ICIC'10. He served as guest editor for Multimedia Tools and Applications, and the International Journal of Semantic Computing. He is a tutorial speaker in ACMMM2011, ISCAS2013, ICIP2013, ICASSP2014, IJCNN2017. He is a Senior Area Editor for \textit{IEEE Transactions on Signal Processing}. He is/was an Associate Editor for \textit{IEEE Transactions on Image Processing, IEEE Transactions on Multimedia, IEEE Transactions on Circuits and Systems for Video Technology, IEEE Transactions on Signal Processing,} and \textit{IEEE Signal Processing Letters}. 
 
\end{IEEEbiography}

\end{document}